\documentclass[12pt]{amsart}
\usepackage{amssymb}
\usepackage{verbatim}
\usepackage{enumerate}
\title{The $*$-variation of the Banach-Mazur game and forcing axioms}
\date{}
\author{Yasuo~Yoshinobu}
\thanks{This work was supported by JSPS Grant-in-Aid for Scientific Research(C) 24540121.}
\address{\newline
Graduate School of Information Science\newline
Nagoya University\newline
Furo-cho, Chikusa-ku, Nagoya 464-8601\newline
JAPAN}
\email{yosinobu@is.nagoya-u.ac.jp}

\newtheorem{dfn}{Definition}[section]

\newtheorem{thm}[dfn]{Theorem}
\newtheorem{lma}[dfn]{Lemma}
\newtheorem{sublma}[dfn]{Sublemma}
\newtheorem{cor}[dfn]{Corollary}

\newcommand{\restrict}{\upharpoonright}

\newcommand{\imply}{{\ \Rightarrow\ }}

\newcommand{\force}{\Vdash}

\newcommand{\concat}[2]{{{#1}^\smallfrown#2}}
\newcommand{\seq}[1]{{\langle#1\rangle}}

\makeatletter
\renewcommand{\p@enumii}{}
\makeatother

\begin{document}

\maketitle
\begin{center}
{\it To the memory of my beloved wife Ayako}
\end{center}

\begin{abstract}
We introduce a property of posets which strengthens $(\omega_1+1)$-strategic closedness. This property is defined using a variation of the Banach-Mazur game on posets, where the first player chooses a countable set of conditions instead of a single condition at each turn. We prove $\mathrm{PFA}$ is preserved under any forcing over a poset with this property. As an application we reproduce a proof of Magidor's theorem about the consistency of $\mathrm{PFA}$ with some weak variations of the square principles. We also argue how different this property is from $(\omega_1+1)$-operational closedness, which we introduced in our previous work, by observing which portions of $\mathrm{MA}^+(\text{$\omega_1$-closed})$ are preserved or destroyed under forcing over posets with either property.
\end{abstract}

\section{Introduction}\label{sec:introduction}

As a part of studies on consequences of various forcing axioms, some studies have been devoted to understanding what kind of forcing preserves those axioms. As one of the earliest comprehensive results in this area, Larson \cite{larson00:_separ} proved that $\mathrm{MM}$ is preserved under forcing over any poset such that every pairwise compatible subset of size at most $\omega_1$ has a common extension. In fact his proof also works for $\mathrm{PFA}$ instead of $\mathrm{MM}$, and for any {\it $\omega_2$-directed\/} closed poset (that is, a poset such that every directed subset of size at most $\omega_1$ has a common extension). As for $\mathrm{PFA}$, K\"onig and the author \cite{koenig04:_fragm_maxim} extended Larson' theorem by showing that it is preserved under forcing over any $\omega_2$-closed poset, although later in \cite{ky12:_kurep_namba} they showed that it is not the case for $\mathrm{MM}$.

Can we still find any reasonable broader class of posets preserving $\mathrm{PFA}$? There are some limitations observed from known results. Caicedo and Velickovic \cite{caicedo_velickovic06:} proved the following theorem.
\begin{thm}[Caicedo and Velickovic \cite{caicedo_velickovic06:}]\label{thm:cv}
Suppose that $\mathrm{BPFA}$ holds both in the universe and in an inner model $M$, and that $M$ computes $\omega_2$ correctly. Then $M$ contains all subsets of $\omega_1$. 
\end{thm}
Note that by Theorem \ref{thm:cv} it is observed that any forcing preserving both $\mathrm{PFA}$ and $\omega_2$ adds no new subsets of $\omega_1$. 

One natural generalization of $\omega_2$-closedness still adding no new subsets of $\omega_1$ is {\it $(\omega_1+1)$-strategic closedness\/}, defined in terms of the existence of a winning strategy for the second player in the corresponding (generalized) Banach-Mazur game of length $(\omega_1+1)$.

$(\omega_1+1)$-strategic closedness is, however, not enough to preserve $\mathrm{PFA}$. In fact, the natural poset adding a $\square_{\omega_1}$-sequence is $(\omega_1+1)$-strategically closed, whereas $\square_{\omega_1}$ fails under $\mathrm{PFA}$ as proved by Todorcevic \cite{PFAnote}.

Considering these facts, one possible approach to obtain a generalization of $\omega_2$-closedness which remain to preserve $\mathrm{PFA}$ is to strengthen the notion of strategic closedness in some appropriate way.

In our previous paper \cite{yoshinobu13:_oper}, we proved that $\mathrm{PFA}$ is preserved under forcing with any {\it operationally closed\/} poset, that is, a poset such that the second player wins the corresponding Banach-Mazur game even when at each turn she is only allowed to use the Boolean infimum of preceding moves and the ordinal number of the turn to decide her move, not allowed to use full information about the preceding moves.

In this paper we introduce another strengthening of $(\omega_1+1)$-strategic closedness preserving $\mathrm{PFA}$ in the following way. We introduce a variation of the Banach-Mazur game where the first player chooses a countable set of conditions at each turn, instead of a single condition. At each moment, the Boolean infimum of all conditions he has chosen by the time plays the same role as his move in the usual Banach-Mazur game. We say a poset is {\it $*$-tactically closed\/} if the second player wins even when at each turn she is only allowed to use the set of conditions the opponent has chosen by the time to decide her move. We also introduce the notion of {\it $*$-operational closedness\/} which generalizes both operational closedness and $*$-tactical closedness, and prove that $\mathrm{PFA}$ is preserved even under forcing with any poset with this property.

As an application of this result we give a proof of the following well-known theorem originally proved by Magidor.

\begin{thm}[Magidor\cite{magidor:_jerusalem}]\normalfont\label{thm:magidor}
$\mathrm{PFA}$ is consistent with the statement that $\square_{\kappa, \omega_2}$ holds for all cardinals $\kappa$ such that $\kappa\geq\omega_2$.
\end{thm}

Since now we have two seemingly different ways to strengthen strategic closedness obtaining the preservation of $\mathrm{PFA}$, it is natural to ask if they are really different. As an answer to this question, we show that under $\mathrm{MA}^+(\text{$\omega_1$-closed})$, another well-known forcing axiom, neither operational nor $*$-tactical closedness implies the other. We do this by producing two consequences of $\mathrm{MA}^+(\text{$\omega_1$-closed})$, $\phi_1$ and $\phi_2$, such that all $*$-tactically closed posets preserve $\phi_1$ but some $*$-tactically closed poset does not preserve $\phi_2$, and all operationally closed posets preserve $\phi_2$ but some operationally closed poset does not preserve $\phi_1$.

This paper is organized as follows. In the rest of this section we introduce and prove some lemmata about forcing which will be used in later sections. In \S \ref{sec:vgBMg} we quickly review the notion of the generalized Banach-Mazur games, and then introduce the $*$-variation of the games and the notion of $*$-tactical closedness of posets. In \S \ref{sec:prePFA} we prove the preservation of $\mathrm{PFA}$ under any $*$-tactically closed forcing, and as its application we give a proof of Theorem \ref{thm:magidor}. In \S \ref{sec:SCP} we introduce a combinatorial principle named as $\mathrm{SCP^-}$ (where $\mathrm{SCP}$ stands for the {\it setwise climbability property\/}), and prove that it is equivalent to $\mathrm{MA}_{\omega_2}$ for $*$-tactically closed posets.
In \S \ref{sec:opstar1}, we show that Chang's Conjecture ($\mathrm{CC}$) holds in any generic extension by any $*$-tactically closed forcing, whenever $\mathrm{MA}^+(\text{$\omega_1$-closed})$ is assumed in the ground model. Since it is known that there exists an operationally closed poset which forces the failure of $\mathrm{CC}$, this result shows that operational closedness does not imply $*$-tactical closedness.
In \S \ref{sec:opstar2}, we show that $\mathrm{SCP^-}$ fails in any generic extension by any operationally closed forcing, whenever $\mathrm{MA}^+(\text{$\omega_1$-closed})$ is assumed in the ground model. Since the natural poset forcing $\mathrm{SCP^-}$ is $*$-tactically closed, this shows that $*$-tactical closedness does not imply operational closedness, either.

Our notation is mostly standard. We adopt the same convention as \cite{yoshinobu13:_oper} for posets: Each of our posets $\mathbb{P}$ is reflexive, transitive and separative (not necessarily antisymmetric), with one greatest element $1_{\mathbb{P}}$ being specified. For a set $A$ of ordinals, $\mathrm{cl}A$ denotes the closure of $A$ and $\mathrm{l.p.}(A)$ denotes the set of limit points of $A$. For $i=0$ or $1$, $S^2_i$ denotes the set $\{\alpha<\omega_2\mid\mathrm{cf}\alpha=\omega_i\}$. We will use the following notations in later sections: For sets $M$, $N$ and an ordinal $\delta<\omega_2$, we denote $M\prec_\delta N$ if $M\prec N$ (as $\in$-structures) and $M\cap\delta=N\cap\delta$ hold. We also denote $M\prec^s_\delta N$ if $M\prec_\delta N$ and $M\cap\omega_2\subsetneq N\cap\omega_2$ hold.

We end this section with introducing a couple of lemmata which will be used in \S \ref{sec:opstar1} and \S \ref{sec:opstar2}.

\begin{dfn}\normalfont\label{dfn:strgen}
Let $\mathbb{P}$ be a poset and $N$ a set.
\begin{enumerate}[(1)]
\item We say $q\in\mathbb{P}$ is {\it $(N, \mathbb{P})$-strongly generic\/} if for every $D\in N$ which is a dense subset of $\mathbb{P}$ there exists an $r\in D\cap N$ such that $r\geq_\mathbb{P}q$.
\item For an $(N, \mathbb{P})$-strongly generic $q$, we denote
$$
q\restrict N:=\bigwedge\{r\in N\cap\mathbb{P}\mid r\geq_\mathbb{P}q\},
$$
where the meet in the right-hand side is computed in the boolean completion $\mathcal{B}(\mathbb{P})$ of $\mathbb{P}$.
\item We say a $\leq_\mathbb{P}$-descending sequence $\seq{p_n\mid n<\omega}$ is {\it $(N, \mathbb{P})$-generic\/} if $p_n\in N$ for every $n<\omega$ and for every dense subset $D\in N$ of $\mathbb{P}$ there exists an $n<\omega$ such that $p_n\in D$. 
\end{enumerate}
\end{dfn}

\begin{lma}\normalfont\label{lma:strgen}
Let $\mathbb{P}$ be a poset, $\theta$ a regular uncountable cardinal such that $\mathbb{P}\in H_\theta$ and $N$ an elementary submodel of $\seq{H_\theta, \in}$ with $\mathbb{P}\in N$. Suppose $\seq{p_n\mid n<\omega}$ is an $(N, \mathbb{P})$-generic sequence and $q$ is a common extension of the $p_n$'s. then $q$ is $(N, \mathbb{P})$-strongly generic and it holds that
$$
q\restrict N=\bigwedge\{p_n\mid n<\omega\}.
$$
\end{lma}

\begin{proof}
By definitions it is easy to see that $q$ is $(N, \mathbb{P})$-strongly generic and $q\restrict N\leq_\mathbb{P}p_n$ for each $n<\omega$.
For each $r\in N\cap\mathbb{P}$ such that $r\geq_\mathbb{P}q$, $$E_r=\{p\in\mathbb{P}\mid p\leq_\mathbb{P}r\lor p\perp_\mathbb{P} r\}$$ is a dense subset of $\mathbb{P}$ and $E_r\in N$ holds, and thus there exists an $n<\omega$ such that $p_n\in E_r$. Since $r$ and $p_n$ has $q$ as a common extension, $p_n\leq_\mathbb{P}r$ must be the case. This gives the required equality.
\end{proof}

Recall that for posets $\mathbb{P}$ and $\mathbb{R}$, a mapping $\pi:\mathbb{R}\to\mathbb{P}$ is said to be a {\it projection\/} if it satisfies the following conditions:
\begin{enumerate}[(a)]
\item $\pi$ is order-preserving.
\item $\pi(1_\mathbb{R})=1_\mathbb{P}$.
\item $\forall r\in\mathbb{R}\forall p\in\mathbb{P}[\pi(r)\geq_\mathbb{P}p\imply\exists r'\leq_{\mathbb{R}}r[p\geq_\mathbb{P}\pi(r')]]$.
\end{enumerate}  

For basic properties of projections see \cite{cummings:_iterated}. Note that, for a projection $\pi:\mathbb{R}\to\mathbb{P}$, whenever $G$ is an $\mathbb{R}$-generic filter over $V$, $\pi''G$ generates a $\mathbb{P}$-generic filter $\pi_*(G)$ over $V$, and $V[G]$ contains $V[\pi_*(G)]$. Knowing this, in later sections we often abusively use each $\mathbb{P}$-name $\tau$ to denote the $\mathbb{R}$-name representing $\tau_{\pi_*(G)}$ in $V[G]$ whenever $G$ is $\mathbb{R}$-generic over $V$.

\begin{lma}\normalfont\label{lma:onestep}
Let $\mathbb{P}$ and $\mathbb{R}$ be posets and $\pi: \mathbb{R}\to\mathbb{P}$ a projection. Suppose $\theta$ is a regular uncountable cardinal such that $\mathbb{P}$, $\mathbb{R}$, $\pi\in H_\theta$, $N$ is a countable elementary submodel of $\seq{H_\theta, \in}$ with $\mathbb{P}$, $\mathbb{R}$, $\pi\in N$, and $q\in\mathbb{P}$ is an $(N, \mathbb{P})$-strongly generic condition. Then for any dense subset $D\in N$ of $\mathbb{R}$ and any $p\in N\cap\mathbb{R}$ satisfying $\pi(p)\geq_\mathbb{P}q$ there exists $p'\leq_\mathbb{R}p$ such that $p'\in N\cap D$ and that $\pi(p')\geq_\mathbb{P}q$.
\end{lma}
\begin{proof}
Note that $D_p=\{r\in D\mid r\leq_\mathbb{R}p\}$ is dense below $p$ and $D_p\in N$ holds. Since $\pi$ is a projection and is in $N$, $\pi''D_p$ is dense below $\pi(p)$ and $\pi''D_p\in N$ holds. Since $q$ is $(N, \mathbb{P})$-strongly generic, there exists an $r\in\pi''D_p\cap N$ such that $r\geq_\mathbb{R}q$. But then again since $\pi\in N$ we have that there exists $p'\in D_P\cap N$ such that $\pi(p')=r$. This $p'$ satisfies all requirements.
\end{proof}

\section{A variation of generalized Banach-Mazur game}\label{sec:vgBMg}

Let us first quickly review some basics of the Banach-Mazur games on posets, introduced by Jech \cite{jech:_game} and generalized by Foreman \cite{foreman83:_games_boolean}. Our notation mostly follows \cite{yoshinobu13:_oper}. For a poset $\mathbb{P}$ and an ordinal $\alpha$, the two-player game $G_\alpha(\mathbb{P})$ is played as follows: Player $\mathrm{I}$ and $\mathrm{II}$ take turns to choose $\mathbb{P}$-conditions one-by-one, so that each move is stronger than all preceding moves. Their turns take place in a well-ordered timeline. Player $\mathrm{I}$ goes first in the beginning of the game, whereas at other limit turns Player $\mathrm{II}$ goes first. Therefore a play of this game can be displayed as follows: 
$$
\begin{matrix}
\mathrm{I}: & a_0 & a_1 & a_2 & \cdots & & a_{\omega+1} & \cdots\\
\mathrm{II}: & \hspace{36pt} b_0 & \hspace{36pt} b_1 & \hspace{36pt} b_2 & \cdots & b_\omega & \hspace{36pt} b_{\omega+1} & \cdots
\end{matrix}
$$
Throughout this paper, we use the above numbering of turns, that is, we consider that Player $\mathrm{I}$ skips his limit turns. Player $\mathrm{II}$ wins this game if she was able to take turns $\alpha$ times, without becoming unable to make legitimate moves on the way. $\alpha$ is called as the {\it length\/} of the game $G_\alpha(\mathbb{P})$. In this paper we are mainly interested in games of length $(\omega_1+1)$.

In each turn of Player $\mathrm{II}$ during a play of $G_\alpha(\mathbb{P})$, we call the sequence $s$ of preceding moves of Player $\mathrm{I}$ (in the chronological order) as the {\it current status\/} of the turn. We call $\bigwedge s$ (computed in $\mathcal{B}(\mathbb{P})$) as the {\it current position\/} of the turn. This terminology makes sense because $\bigwedge s$ gives the exact upper bound of possible moves of Player $\mathrm{II}$ of the turn. We also call the order type of preceding moves of Player $\mathrm{II}$ as the {\it ordinal number\/} of the turn.

A {\it strategy\/} (of Player $\mathrm{II}$ for $G_\alpha(\mathbb{P})$) is a function which, in each turn of Player $\mathrm{II}$, takes the current status of the turn as an argument, and gives a $\mathbb{P}$-condition as a suggestion for Player $\mathrm{II}$'s move of the turn. A strategy $\sigma$ is a {\it winning strategy\/} if Player $\mathrm{II}$ wins any play of $G_\alpha(\mathbb{P})$ as long as she plays as $\sigma$ suggests.

A strategy is called an {\it operation\/} ({\it resp. tactic\/}) if its suggestion depends only on the current position and the ordinal number ({\it resp. only on the current position\/}) of each turn.

We say $\mathbb{P}$ is {\it $\alpha$-strategically\/} ({\it resp. $\alpha$-operationally, $\alpha$-tactically\/}) {\it closed\/} if there exists a winning strategy ({\it resp.\/} operation, tactic).

The following lemma is frequently used in later sections.

\begin{lma}\normalfont\label{lma:gseqstrcl}
Suppose $\mathbb{P}$ is $(\omega+1)$-strategically closed poset, $\tau$ is a winning strategy for $G_{\omega+1}(\mathbb{P})$, $\theta$ is a regular cardinal such that $\mathbb{P}\in H_\theta$ and $\seq{p_n\mid n<\omega}$ is an $(N, \mathbb{P})$-generic sequence for a countable $N\prec\seq{H_\theta, \in, \mathbb{P}, \tau}$. Then $p_n$'s have a common extension in $\mathbb{P}$.
\end{lma}

\proof We may assume that $p_n$'s are strictly increasing. We will define a strictly increasing sequence $\seq{i_n\mid n<\omega}$ of natural numbers and a sequence $\seq{q_n\mid n<\omega}$ of conditions of $\mathbb{P}\cap N$ as follows. Since $D_0=\{\tau(\seq{q})\mid q\in\mathbb{P}\}$ is dense in $\mathbb{P}$ and is definable in $\seq{H_\theta, \in, \mathbb{P}, \tau}$ and thus lies in $N$, we may pick an $i_0<\omega$ so that $p_{i_0}\in D_0$. Since $p_{i_0}\in N$ we may also pick $q_0\in N\cap\mathbb{P}$ such that $p_{i_0}=\tau(\seq{q_0})$. Note that $\seq{q_0, p_{i_0}}$ forms a part of a play of $G_{\omega+1}(\mathbb{P})$ where Player $\mathrm{II}$ plays as $\tau$ suggests. Now assume $\seq{q_0, p_{i_0}, \ldots, q_n, p_{i_n}}\in N$ was defined and forms a part of a play of $G_{\omega+1}(\mathbb{P})$ where Player $\mathrm{II}$ plays as $\tau$ suggests. Since $D_{n+1}=\{\tau(\seq{q_0, q_1, \ldots, q_n, q})\mid q\leq_\mathbb{P}p_{i_n+1}\}$ is dense below $p_{i_n+1}$ and lies in $N$, we may pick an $i_{n+1}>i_n$ and $q_{n+1}\leq_\mathbb{P}p_{i_n}$ such that $q_{n+1}\in N$ and $\tau(\seq{q_0, q_1, \ldots, q_n, q_{n+1}})=p_{i_{n+1}}$. This assures that $\seq{q_0, p_{i_0}, \ldots, q_n, p_{i_n}, q_{n+1}, p_{i_{n+1}}}$ is in $N$ and forms a part of a play of $G_{\omega+1}(\mathbb{P})$ where Player $\mathrm{II}$ plays as $\tau$ suggests, and thus the construction goes on. In the end we have that $\seq{q_n, p_{i_n}\mid n<\omega}$ forms a part of a play of $G_{\omega+1}(\mathbb{P})$ where Player $\mathrm{II}$ plays as $\tau$ suggests, and thus has a common extension in $\mathbb{P}$. Since $\seq{p_{i_n}\mid n<\omega}$ is cofinal in $\seq{p_n\mid n<\omega}$, the latter also has a common extension in $\mathbb{P}$. \qed

Now we introduce a variation of $G_{\omega_1+1}(\mathbb{P})$ to define new game closedness properties of posets.

\begin{dfn}\normalfont
For a poset $\mathbb{P}$, $G^*(\mathbb{P})$ denotes the following game: Players take turns in the same way as in $G_{\omega_1+1}(\mathbb{P})$, but Player $\mathrm{I}$ chooses a subset of $\mathbb{P}$ of size at most countable at each turn, instead of a single condition.
Therefore a play of $G^*(\mathbb{P})$ can be displayed as follows:

$$
\begin{matrix}
\mathrm{I}: & A_0  & A_1 & A_2 & \cdots & & A_{\omega+1} & \cdots\phantom{,}\\
\mathrm{II}: & \hspace{36pt} b_0 & \hspace{36pt} b_1 & \hspace{36pt} b_2 & \cdots & b_\omega & \hspace{36pt} b_{\omega+1} & \cdots,
\end{matrix}
$$
where $A_\gamma\in[\mathbb{P}]^{\leq\aleph_0}\setminus\{\emptyset\}$ for each $\gamma\in\omega_1\setminus\mathrm{Lim}$, and $b_\gamma\in\mathbb{P}$ for each $\gamma<\omega_1$. They must obey the following rules (Player $\mathrm{I}$ is responsible for (\ref{cond:gamerulea})--(\ref{cond:gamerulec}) and Player $\mathrm{II}$ is for (\ref{cond:gameruled}):
\begin{enumerate}[(a)]
\item $\seq{A_\gamma\mid\gamma\in\omega_1\setminus\mathrm{Lim}}$ is $\subseteq$-increasing.\label{cond:gamerulea}
\item For each $\gamma\in\omega_1\setminus\mathrm{Lim}$, $A_\gamma$ has a common extension in $\mathbb{P}$.\label{cond:gameruleb}
\item For each $\gamma<\omega_1$, it holds that $\bigwedge A_{\gamma+1}\leq_{\mathcal{B}(\mathbb{P})}b_\gamma$.\label{cond:gamerulec}
\item For each $\gamma<\omega_1$, $b_\gamma$ is a common extension of $A_\gamma$ (for limit $\gamma$ we define $A_\gamma=\bigcup\{A_\xi\mid\xi\in\gamma\setminus\mathrm{Lim}\}$).\label{cond:gameruled}
\end{enumerate}
Again, Player $\mathrm{II}$ wins this game if she was able to make her $\omega_1$-th move.
\end{dfn}

Note that, in the above play of $G^*(\mathbb{P})$, if we replace each $A_\gamma$ with its boolean infimum, we will obtain a play of $G_{\omega_1+1}(\mathcal{B}(\mathbb{P}))$. In fact, as long as players play with perfect recall, $G^*(\mathbb{P})$ is essentially an equivalent game to $G_{\omega_1+1}(\mathcal{B}(\mathbb{P}))$, and even to $G_{\omega_1+1}(\mathbb{P})$. The point of introduction of $G^*(\mathbb{P})$ lies in cases when players have limited memory on preceding moves.

\begin{dfn}\normalfont
Let $\mathbb{P}$ be a poset. We let $[\mathbb{P}]^{\leq\aleph_0}_+$ denote the set $\{A\in[\mathbb{P}]^{\leq\aleph_0}\setminus\{\emptyset\}\mid\text{$A$ has a common extension in $\mathbb{P}$}\}$.
\begin{enumerate}[(1)]
\item A function from $\omega_1\times[\mathbb{P}]^{\leq\aleph_0}_+\to\mathbb{P}$ is called a {\it $*$-operation\/} for $G^*(\mathbb{P})$. In a play of $G^*(\mathbb{P})$, we say Player $\mathrm{II}$ {\it plays according to\/} a $*$-operation $\tau$ if, for each $\delta<\omega_1$ she chooses $\tau(\delta, A_\delta)$ as her $\delta$-th move as long as it is a legal move, where $A_\delta$ denotes the union of subsets of $\mathbb{P}$ chosen by Player $\mathrm{I}$ by the time. $\tau$ is said to be a {\it winning $*$-operation\/} for $G^*(\mathbb{P})$ if Player $\mathrm{II}$ wins any play of $G^*(\mathbb{P})$ as long as she plays according to $\tau$. $\mathbb{P}$ is {\it $*$-operationally closed \/} if there exists a winning $*$-operation for $G^*(\mathbb{P})$.
\item A $*$-operation $\tau$ for $G^*(\mathbb{P})$ is called a {\it $*$-tactic\/} if the values of $\tau$ do not depend on its first argument. We often consider a $*$-tactic simply as a function defined on $[\mathbb{P}]^{\leq\aleph_0}_+$. $\mathbb{P}$ is {\it $*$-tactically closed\/} if there exists a winning $*$-tactic for $G^*(\mathbb{P})$. 
\end{enumerate}
\end{dfn}

A nontrivial example of a $*$-tactically closed forcing is found among the class of natural posets forcing the following combinatorial principles introduced by Schimmerling \cite{schimmerling:_onewoodin} which are weaker variations of Jensen's square principles.

\begin{dfn}\normalfont\label{dfn:wsq}
For an uncountable cardinal $\kappa$ and a cardinal $\lambda$ with $1\leq\lambda\leq\kappa$, $\square_{\kappa, \lambda}$ denotes the following statement: There exists a sequence $\vec{\mathcal{C}}=\seq{\mathcal{C}_\alpha\mid\alpha\in\kappa^+\cap\mathrm{Lim}}$ such that for every limit $\alpha<\kappa^+$

\begin{enumerate}[(i)]
  \item $1\leq|\mathcal{C}_\alpha|\leq\lambda$.\label{cond1}
  \item $\mathcal{C}_\alpha$ consists of club subsets of $\alpha$ of order type $\leq\kappa$.\label{cond2}
  \item For every $C\in\mathcal{C}_\alpha$ and $\beta\in\mathrm{l.p.}(C)$, $C\cap\beta\in\mathcal{C}_\beta$ holds.\label{cond3}
\end{enumerate}
\end{dfn}

Note that $\square_{\kappa, 1}$ is the original Jensen's square principle $\square_\kappa$.

\begin{dfn}\normalfont\label{dfn:pskl}
Let $\kappa$ and $\lambda$ be as in Definition \ref{dfn:wsq}. $\mathbb{P}_{\square_{\kappa, \lambda}}$ denotes the following poset. A condition $p\in\mathbb{P}_{\square_{\kappa, \lambda}}$ is either the empty sequence $1_{\mathbb{P}_{\square_{\kappa, \lambda}}}=\seq{}$, or of the form
\begin{equation}\label{eqn:pform}
p=\seq{\mathcal{C}^p_\alpha\mid\alpha\in(\alpha^p+1)\cap\mathrm{Lim}}\quad\text{(where $\alpha^p\in\kappa^+\cap\mathrm{Lim}$)}
\end{equation}
which satisfies (\ref{cond1})--(\ref{cond3}) of Definition \ref{dfn:wsq} for every limit $\alpha\leq\alpha^p$. $\mathbb{P}_{\square_{\kappa, \lambda}}$ is ordered by initial segment.
\end{dfn}

Note that, since $\mathbb{P}_{\square_{\kappa, \lambda}}$ is $(\kappa+1)$-strategically closed, it preserves cardinalities below $\kappa^+$ and thus forces $\square_{\kappa, \lambda}$.

\begin{thm}\label{thm:wsq}
Let $\kappa$ be an uncountable cardinal, and $\lambda$ a cardinal satisfying $2^{\aleph_0}\leq\lambda\leq\kappa$. Then $\mathbb{P}_{\square_{\kappa, \lambda}}$ is $*$-tactically closed.
\end{thm}

\proof We define a $*$-tactic $\tau:[\mathbb{P}_{\square_{\kappa, \lambda}}]^{\leq\aleph_0}_+\to\mathbb{P}$ as follows. Let $A\in[\mathbb{P}_{\square_{\kappa, \lambda}}]^{\leq\aleph_0}_+$ be arbitrary.

\begin{enumerate}[(1)]
\item If $A$ has a strongest condition $p$, if $p$ is of the form as in (\ref{eqn:pform}) of Definition \ref{dfn:pskl}, then we let
$$
\tau(A):=\concat{p}{\seq{\mathcal{C}^p_{\alpha^p+\omega}}}\ \text{(where $\mathcal{C}^p_{\alpha^p+\omega}:=\{\{\alpha^p+n\mid n<\omega\}\}$).}
$$
If $p=1_{\mathbb{P}_{\square_{\kappa, \lambda}}}$ we let $\tau(A):=\seq{\mathcal{C}^p_\omega}\ \text{(where $\mathcal{C}^p_\omega=\{\omega\}$)}$.

In both cases it is easy to check that $\tau(A)$ forms a $\mathbb{P}_{\square_{\kappa, \lambda}}$-condition which extends $p$.

\item\label{case2} If $A$ has no strongest condition, then $q:=\bigcup A$ is of the form
$$
q=\seq{\mathcal{C}^q_\alpha\mid\alpha\in\beta\cap\mathrm{Lim}} \quad\text{for some limit $\beta$}.
$$
Let $a=\{\gamma\in\beta\cap\mathrm{Lim}\mid q\restrict(\gamma+1)\in A\}$. Then $a$ is a countable subset of $\beta\cap\mathrm{Lim}$ which is unbounded in $\beta$. In this case we set $\tau(A):=\concat{q}{\seq{\mathcal{C}^q_{\beta}}}$, where
$$
\mathcal{C}^q_{\beta}:=\{C\subseteq_{\text{club}}\beta\cap\mathrm{cl}(a)\mid\forall\gamma\in\beta\cap\mathrm{l.p.}(C)[C\cap\gamma\in\mathcal{C}_\gamma]\}.
$$

Note that $|\mathcal{C}^q_{\beta}|\leq 2^{\aleph_0}$ holds since $\mathrm{cl}(a)$ is countable. On the other hand, $\mathcal{C}^q_{\beta}$ is nonempty since $\beta\cap\mathrm{cl}(a)$ has a subset of order type $\omega$ which is unbounded in $\beta$. Now it is easy to check that $\tau(A)$ forms a $\mathbb{P}_{\square_{\kappa, \lambda}}$-condition which extends all conditions in $A$.
\end{enumerate}

Let us show that $\tau$ is a winning $*$-tactic. Consider any play of the game $G^*(\mathbb{P}_{\square_{\kappa, \lambda}})$ where Player $\mathrm{II}$ plays according to $\tau$. Note that since $\mathbb{P}_{\square_{\kappa, \lambda}}$ is $\omega_1$-closed, the play never ends throughout the first $\omega_1$ turns of both players. Let us display the play as follows:
$$
\begin{matrix}
\mathrm{I}: & A_0  & A_1 & A_2 & \cdots & & A_{\omega+1} & \cdots\phantom{,}\\
\mathrm{II}: & \hspace{36pt} q_0 & \hspace{36pt} q_1 & \hspace{36pt} q_2 & \cdots & q_\omega & \hspace{36pt} q_{\omega+1} & \cdots.
\end{matrix}
$$
Let us also set $A_\gamma:=\bigcup\{A_\xi\mid\xi\in\gamma\setminus\mathrm{Lim}\}$ for each limit $\gamma<\omega_1$. Note that then $q_\gamma=\tau(A_\gamma)$ holds for every $\gamma<\omega_1$. It is enough to show that the $q_\gamma$'s have a common extension in $\mathbb{P}_{\square_{\kappa, \lambda}}$.
Since $\seq{q_\gamma\mid\gamma<\omega_1}$ is a descending sequence in $\mathbb{P}_{\square_{\kappa, \lambda}}$, $q:=\bigcup\{q_\gamma\mid\gamma<\omega_1\}$ is of the form 
$$
q=\seq{\mathcal{C}^q_\alpha\mid\alpha\in\beta^q\cap\mathrm{Lim}}\quad\text{for some limit $\beta^q<\kappa^+$,}
$$
and for each $\gamma<\omega_1$, $q_\gamma=q\restrict(\beta^q_\gamma+1)$ for some limit $\beta^q_\gamma<\beta^q$. By the definition of $\tau$, it is easy to observe that $\seq{\beta^q_\gamma\mid\gamma<\omega_1}$ is continuous, strongly increasing, and thus we have $\beta^q=\sup\{\beta^q_\gamma\mid\gamma<\omega_1\}$ and $\mathrm{cf}\beta^q=\omega_1$.

So it is enough to show that there exists a club subset $C$ of $\beta$ such that $\mathrm{o.t.}(C)=\omega_1$ and $C\cap\alpha\in\mathcal{C}^q_\alpha$ holds for each $\alpha\in\mathrm{l.p.}(C)$, because if it is the case, then
$$
\concat{q}{\seq{\mathcal{C}^q_{\beta^q}}}\ \text{(where $\mathcal{C}^q_{\beta^q}:=\{C\}$)}
$$ 
forms a $\mathbb{P}_{\square_{\kappa, \lambda}}$-condition which extends all $q_\gamma$'s. For each $\gamma<\omega_1$, pick a subset $a_\gamma$ of $(\beta^q_{\omega\gamma}, \beta^q_{\omega(\gamma+1)})\cap\mathrm{Lim}$ such that $a_\gamma$ is unbounded in $\beta^q_{\omega(\gamma+1)}$, $\mathrm{o.t.}(a_\gamma)=\omega$ and that $q\restrict(\xi+1)\in A_{\omega(\gamma+1)}$ for every $\xi\in a_\gamma$ (this is possible because $A_{\omega(\gamma+1)}$ has no strongest condition). Then let
$$
C:=\bigcup\{a_\gamma\mid\gamma<\omega_1\}\cup\{\beta^q_\gamma\mid\gamma\in\omega_1\cap\mathrm{Lim}\}.
$$
It is easy to see that $C$ is a club subset of $\beta^q$ of order type $\omega_1$, and that any limit point of $C$ other than $\beta^q$ is of the form $\beta^q_\gamma$ for some limit $\gamma<\omega_1$. 
It is also easy to see that, for each limit $\gamma<\omega_1$,
$$
C\cap\beta^q_\gamma=\bigcup\{a_\delta\mid\delta<\gamma\}\cup\{\beta^q_\delta\mid\delta\in\gamma\cap\mathrm{Lim}\}
$$
is unbounded in $\beta^q_\gamma$ and is included by $\mathrm{cl}(\{\xi\in\beta^q_\gamma\cap\mathrm{Lim}\mid q\restrict(\xi+1)\in A_\gamma\})$.
Now we have $C\cap\beta^q_\gamma\in\mathcal{C}^q_{\gamma}$ for every limit $\gamma<\omega_1$, by induction on $\gamma$.\qed{(Theorem \ref{thm:wsq})}

\bigskip
The following iteration lemma will be used in the next section.
\begin{lma}\normalfont\label{lma:iteration}
Suppose $\seq{\mathbb{P}_\alpha, \dot{\mathbb{Q}}_\beta\mid \alpha\leq\gamma, \beta<\gamma}$ is an iterated forcing construction such that:
\begin{enumerate}[(i)]
\item $\mathbb{P}_0$ is a trivial poset.
\item For each $\alpha<\gamma$ it holds that
$\force_{\mathbb{P}_\alpha}\text{\lq\lq$\dot{\mathbb{Q}}_\alpha$ is $*$-tactically closed.\rq\rq}$
\item For each limit ordinal $\xi\leq\gamma$, $\mathbb{P}_\xi$ is either the direct limit or the inverse limit of $\seq{\mathbb{P}_\zeta\mid\zeta<\xi}$, and the latter is the case whenever $\mathrm{cf}(\xi)\leq\omega_1$ holds.
\end{enumerate}
Then $\mathbb{P}_\gamma$ is $*$-tactically closed.
\end{lma}

For proof of Lemma \ref{lma:iteration}, we use a sublemma stated below.

\begin{dfn}\normalfont\label{dfn:defensive}
Let $\mathbb{P}$ be a poset. A $*$-tactic $\tau$ for $G^*(\mathbb{P})$ is said to be {\it defensive\/} if $\tau(\{1_{\mathbb{P}}\})=1_{\mathbb{P}}$ holds.
\end{dfn}

\begin{sublma}\normalfont\label{sublma:defensivewinning}
Every $*$-tactically closed poset has a defensive winning $*$-tactic.
\end{sublma}

\begin{proof}
Suppose $\tau$ is a winning $*$-tactic for $G^*(\mathbb{P})$. Define $\tau'$ as follows.
$$
\begin{cases}
\tau'(\{1_\mathbb{P}\})&=1_\mathbb{P},\\
\tau'(A)&=\tau(A) \quad\text{if $A\in[\mathbb{P}]^{\leq\omega}_+\setminus\{\{1_\mathbb{P}\}\}$}.
\end{cases}
$$
Consider any play of $G^*(\mathbb{P})$ where Player $\mathrm{II}$ plays according to $\tau'$. As long as Player $\mathrm{I}$ keeps choosing $\{1_\mathbb{P}\}$  as his moves, Player $\mathrm{II}$ keeps choosing $1_\mathbb{P}$ as her moves. Once Player $\mathrm{I}$ chooses any other move, the remaining game will develop as if the game starts from that turn and Player $\mathrm{II}$ plays according to $\tau$. Therefore $\tau'$ is a winning $*$-tactic, which is also defensive.
\end{proof}

\proof[Proof of Lemma \ref{lma:iteration}] Let $D:=\{\alpha\leq\gamma\mid\text{$\mathbb{P}_\alpha$ is the direct limit of $\{\mathbb{P}_\zeta\mid\zeta<\alpha\}$}\}$. For each $\alpha<\gamma$, fix a $\mathbb{P}_\alpha$-name $\dot{\tau}_\alpha$   such that
$$
\force_{\mathbb{P}_\alpha}\text{\lq\lq$\dot{\tau}_\alpha$ is a defensive winning $*$-tactic for $G^*(\dot{\mathbb{Q}}_\alpha)$.\rq\rq}
$$
Now we define $\tau$ as follows. For each $A\in[\mathbb{P}_\gamma]^{\leq\omega}_+$ and $\alpha<\gamma$, let $(\tau(A))(\alpha)$ be a $\mathbb{P}_\alpha$-name satisfying
$$
\force_{\mathbb{P}_\alpha}\text{\lq\lq$(\tau(A))(\alpha)\in\dot{\mathbb{Q}}_\alpha$\rq\rq}
$$
and
$$
\force_{\mathbb{P}_\alpha}\text{\lq\lq$\{p(\alpha)\mid \check{p}\in \check{A}\}\in[\dot{\mathbb{Q}}_\alpha]^{\leq\omega}_+\imply(\tau(A))(\alpha)=\dot{\tau}_\alpha(\{p(\alpha)\mid \check{p}\in \check{A}\})$.\rq\rq}
$$
Note that $(\tau(A))(\alpha)=1_{\dot{\mathbb{Q}}_\alpha}$ holds whenever $p(\alpha)=1_{\dot{\mathbb{Q}}_\alpha}$ for all $p\in A$ and therefore the support of $\tau(A)$ is equal to the union of supports of members of $A$. This assures that $\tau(A)\in\mathbb{P}_\gamma$, since for each $\alpha\in D$, $\mathrm{cf}(\alpha)>\omega_1$ holds and thus the support of $\tau(A)$ is bounded below $\alpha$. Consider any play of $G^*(\mathbb{P}_\gamma)$ where Player $\mathrm{II}$ plays according to $\tau$. At each limit turn, suppose $A$ denotes the set of conditions chosen by Player $\mathrm{I}$ so far. We can construct a common extension $p\in\mathbb{P}_\gamma$ of $A$ as follows. Suppose $p\restrict\alpha\in\mathbb{P}_\alpha$ is already defined so that it extends $q\restrict\alpha$ for all $q\in A$. Then let $p(\alpha)$ be a $\mathbb{P}_\alpha$-name for a $\dot{\mathbb{Q}}_\alpha$-condition satisfying
\begin{enumerate}[(i)]
\item $p\restrict\alpha\force_{\mathbb{P}_\alpha}\text{\lq\lq$p(\alpha)\leq_{\dot{\mathbb{Q}}_\alpha}q(\alpha)$\rq\rq}$ for all $q\in A$, and
\item $p(\alpha)=1_{\dot{\mathbb{Q}}_\alpha}$, if $q(\alpha)=1_{\dot{\mathbb{Q}}_\alpha}$ holds for all $q\in A$.
\end{enumerate}
It is possible by our definition of $\dot{\tau}_\alpha$ and $\tau$, and we have that $p\restrict\alpha+1\in\mathbb{P}_{\alpha+1}$ extends $q\restrict\alpha+1$ for all $q\in A$. This construction assures that $p\restrict\alpha\in\mathbb{P}_\alpha$ for $\alpha\in D$, because $|A|\leq\omega_1$ holds (note that $|A|=\omega_1$ can be the case at the $\omega_1$-th turn of Player $\mathrm{II}$) and the support of $p\restrict\alpha$ is equal to the union of supports of $q\restrict\alpha$ for $q\in A$ and thus is bounded since $\mathrm{cf}(\alpha)>\omega_1$. Thus the play is never stopped on the way of the game, and therefore we have that $\tau$ is a winning $*$-tactic.\qed

We conjecture that the analogue of Lemma \ref{lma:iteration} for $*$-operationally closed posets is also valid, but at present we have no proof, because so far we do not know if the analogue of Sublemma \ref{sublma:defensivewinning} for $*$-operations is correct.


\section{Preservation of $\mathrm{PFA}$}\label{sec:prePFA}
In this section we prove the following theorem.
\begin{thm}\normalfont\label{thm:PFApreserved}
$\mathrm{PFA}$ is preserved under any $*$-operationally closed forcing.
\end{thm}
Note that Theorem \ref{thm:PFApreserved} generalizes \cite[Theorem 10]{yoshinobu13:_oper} which claims that $\mathrm{PFA}$ is preserved under any operationally closed forcing, and since the basic structure of our proof is the same as the one given there, we will expose our proof somewhat briefly, rather focusing on differences from the older proof.

\proof
Suppose $\mathrm{PFA}$ holds in $V$. Let $\mathbb{P}$ be any $*$-operationally closed poset, and $\sigma$ a winning $*$-operation for $G^*(\mathbb{P})$. Without loss of generality, we may assume that $\sigma(\gamma, A)\leq_{\mathcal{B}(\mathbb{P})}\bigwedge A$ holds for every $\gamma<\omega_1$ and $A\in[\mathbb{P}]^{\leq\aleph_0}_+$. We will show that $\mathrm{PFA}$ remains true in $V^\mathbb{P}$. Let $\dot{\mathbb{Q}}$ be any $\mathbb{P}$-name for a proper poset,
and $\seq{\dot{D}_\xi\mid\xi<\omega_1}$ $\mathbb{P}$-names for a dense subset of $\dot{\mathbb{Q}}$.
It is enough to show that there exists a $\mathbb{P}$-name $\dot{F}$ such that
\begin{equation*}
\force_{\mathbb{P}}\text{\lq\lq$\dot{F}$ is a filter on $\dot{\mathbb{Q}}$ and $\dot{F}\cap\dot{D}_\xi\not=\emptyset$ for every $\xi<\omega_1$.\rq\rq}
\end{equation*}

For any $\mathbb{P}$-generic filter $G$ over $V$ and any $\dot{\mathbb{Q}}_G$-generic filter $H$ over $V[G]$, we define a poset $\mathbb{R}=\mathbb{R}_{G, H}$ within $V[G][H]$ as follows:
A condition of $\mathbb{R}$ is either the empty sequence $1_{\mathbb{R}}=\seq{}$, or of the form $P=\seq{A^P_\xi\mid\xi\leq\alpha^P}$ ($\alpha^P<\omega_1$) such that
\begin{enumerate}[(a)]
\item $P$ is a $\subseteq$-continuous increasing sequence of elements of $[\mathbb{P}]^{\leq\aleph_0}_+\cap V$.
\item $\sigma(\xi, A^P_\xi)\geq_{\mathcal{B}(\mathbb{P})}\bigwedge A^P_{\xi+1}$ for every $\xi<\alpha^P$.
\item $\sigma(\alpha^P, A^P_{\alpha^P})\in G$.
\end{enumerate}
For this $P$ we write $\alpha^P$ and $A^P_{\alpha^P}$ as $l(P)$ and $A^P$ respectively.

$\mathbb{R}$ is ordered by initial segment. Let $\dot{\mathbb{R}}$ denote the canonical $(\mathbb{P}*\dot{\mathbb{Q}})$-name representing $\mathbb{R}_{G, H}$.

\medskip\noindent
\underline{Claim} $\mathbb{S}:=\mathbb{P}*\dot{\mathbb{Q}}*\dot{\mathbb{R}}$ is proper.

\proof[Proof of Claim]
Let $\theta$ be a regular cardinal such that $\mathbb{P}$, $\dot{\mathbb{Q}}$, $\dot{\mathbb{R}}$, $\mathbb{S}\in H_\theta$ and $N$ an arbitrary countable elementary submodel of $\seq{H_\theta, \in, \{\mathbb{P}, \dot{\mathbb{Q}}, \dot{\mathbb{R}}, \mathbb{S}\}}$. Set $\delta:=N\cap\omega_1$. Let $\seq{p, \dot{q}, \dot{P}}\in{\mathbb{S}}\cap N$ be arbitrary. It is enough to show that there exists an $(N, \mathbb{P})$-generic $p'\leq_{\mathbb{P}}p$ such that whenever $G$ is a $\mathbb{P}$-generic filter over $V$ with $p'\in G$, there exists an $(N[G], \dot{\mathbb{Q}}_G)$-generic $q'\leq_{\dot{\mathbb{Q}}_G}\dot{q}_G$ satisfying the following: Whenever $H$ is a $\dot{\mathbb{Q}}_G$-generic filter over $V[G]$ with $q'\in H$, there exists an $(N[G][H], \dot{\mathbb{R}}_{G*H})$-generic $P'\leq_{\dot{\mathbb{R}}_{G*H}}\dot{P}_{G*H}$.

First pick an $(N, \mathbb{P})$-generic sequence $\seq{p_n\mid n<\omega}$. By Lemma \ref{lma:gseqstrcl}, $p_n$'s have a common extension in $\mathbb{P}$. Set $A:=\{s\in N\cap\mathbb{P}\mid\exists n<\omega[s\geq_{\mathbb{P}}p_n]\}$. Then $A\in[\mathbb{P}]^{\leq\aleph_0}_+$. Set $p':=\sigma(\delta, A)$.

Suppose $G$ is any $\mathbb{P}$-generic filter over $V$ with $p'\in G$.  Since $p'$ extends $\bigwedge A$ and thus all $p_n$'s, by the $(N, \mathbb{P})$-genericity of $\seq{p_n\mid n<\omega}$ it is easy to see that
\begin{equation}\label{eqn:AisNcapG}
A=N\cap G.
\end{equation}
Since $\dot{q}_G\in N[G]$ and $\dot{\mathbb{Q}}_G$ is proper in $V[G]$, we can pick an $(N[G], \dot{\mathbb{Q}}_G)$-generic $q'\leq_{\dot{\mathbb{Q}}_G}\dot{q}_G$. Now suppose $H$ is any $\dot{\mathbb{Q}}_G$-generic filter over $V[G]$ with $q'\in H$. Since $p'\in G$ is $(N, \mathbb{P})$-generic and $q'\in H$ is $(N[G], \dot{\mathbb{Q}}_G)$-generic, we have $N[G][H]\cap V=N$ and in particular $N[G][H]\cap\omega_1=\delta$. Since $\dot{P}_{G*H}\in N[G][H]$, it is possible to pick an $(N[G][H], \dot{\mathbb{R}}_{G*H})$-generic sequence $\seq{P_n\mid n<\omega}$ with $P_0=\dot{P}_{G*H}$. Set $\tilde{P}:=\bigcup\{P_n\mid n<\omega\}$. Again by an easy density argument we have that $\tilde{P}$ is of the form $\seq{A_\gamma\mid\gamma<\delta}$.

\medskip\noindent
\underline{Subclaim} $\bigcup\{A_\gamma\mid\gamma<\delta\}=A$.

\proof[Proof of Subclaim]
Since $P_n\in N[G][H]$ for each $n<\omega$ and $\delta\subseteq N[G][H]$, we have $A_\gamma\in N[G][H]$ for each $\gamma<\delta$. Note that, for each $\gamma<\delta$, $A_\gamma$ is countable and is in $V$, and thus is contained in $N$ since $N[G][H]\cap V=N$ holds. Moreover $A_\gamma\subseteq G$ also holds for each $\gamma<\delta$ by the definition of $\dot{\mathbb{R}}$. Thus by (\ref{eqn:AisNcapG}) we have $\bigcup\{A_\gamma\mid\gamma<\delta\}\subseteq A$. For the inclusion of the other direction, by (\ref{eqn:AisNcapG}) and the $(N[G][H], \dot{\mathbb{R}}_{G*H})$-genericity of the sequence $\seq{P_n\mid n<\omega}$, it is enough to show the following density lemma:
\begin{lma}\normalfont\label{lma:dadense}
In $V[G][H]$, $D_a=\{R\in\dot{\mathbb{R}}_{G*H}\mid a\in A^R\}$ is dense in $\dot{\mathbb{R}}_{G*H}$ for each $a\in G$.
\end{lma}
\proof[Proof of Lemma \ref{lma:dadense}]
Suppose $R\in\dot{\mathbb{R}}_{G*H}$ and $a\in G$. By the definition of $\dot{\mathbb{R}}_{G*H}$ it holds that $\sigma(l(R), A^R)\in G$. Therefore $a$ and $\sigma(l(R), A^R)$ are compatible, and since $\sigma(l(R)+1, A^R\cup\{a, b\})\leq_{\mathbb{P}}b$ holds for each common extension $b$ of $a$ and $\sigma(l(R), A^R)$ we have
$$
\{\sigma(l(R)+1, A^P\cup\{a, b\})\mid b\leq_{\mathbb{P}}a\ \land\ b\leq_{\mathbb{P}}\sigma(l(R), A^R)\}
$$
is dense below the boolean meet of $a$ and $\sigma(l(R), A^R)$. But this set is defined in $V$, and thus by the genericity of $G$ there exists some common extension $b$ of $a$ and $\sigma(l(R), A^R)$ such that $\sigma(l(R)+1, A^R\cup\{a, b\})\in G$. Therefore we have $R':=\concat{R}{(A^R\cup\{a, b\})}\in\dot{\mathbb{R}}_{G*H}$ and thus $R'\in D_a$. \qed(Lemma \ref{lma:dadense})

\qed(Subclaim)

Since $\sigma(\delta, A)=p'$ belongs to $G$, by letting $A_\delta:=A$ and $P':=\seq{A_\gamma\mid\gamma\leq\delta}$ we have that $P'$ is an $\dot{\mathbb{R}}_{G*H}$-condition which extends all $P_n$'s, and thus is an $(N[G][H], \dot{\mathbb{R}}_{G*H})$-generic condition. This finishes the proof of our claim.\qed(Claim)

Note that, by the same density argument as above, for any ${\mathbb{S}}$-generic filter $G*H*I$ over $V$, $\bigcup I$ is an $\omega_1$-sequence of elements of $[\mathbb{P}]^{\leq\aleph_0}_+\cap V$. For each $\gamma<\omega_1$ let $\dot{A}_\gamma$ be the canonical ${\mathbb{S}}$-name representing the $\gamma$-th entry of this sequence.

In $V$, by applying $\mathrm{PFA}$ to ${\mathbb{S}}$ with a sufficiently rich family of $\aleph_1$-many dense subsets of ${\mathbb{S}}$ we can find a directed subset $\{s_\gamma=\seq{p_\gamma, \dot{q}_\gamma, \dot{P}_\gamma}\mid\gamma<\omega_1\}$ of ${\mathbb{S}}$, $\{A_\gamma\mid\gamma<\omega_1\}\subseteq[\mathbb{P}]^{\leq\aleph_0}_+$ and $\{\xi_\gamma\mid\gamma<\omega_1\}\subseteq\omega_1$ such that for each $\gamma<\omega_1$ it holds that
\begin{equation}
  s_\gamma\force_{{\mathbb{S}}}\text{\lq\lq$\dot{A}_\gamma=\check{A}_\gamma\land l(\dot{P}_\gamma)=\check{\xi_\gamma}\land\sigma(\check{\xi_\gamma}, A^{\dot{P}_\gamma})=\check{p_\gamma}$\rq\rq}\label{sforces}
\end{equation}
and
\begin{equation}
  p_\gamma\force_{\mathbb{P}}\text{\lq\lq$\dot{q}_\gamma\in\dot{D}_\gamma$.\rq\rq}\label{pforces}
\end{equation}
By the directedness and genericity of the set $\{s_\gamma\mid\gamma<\omega_1\}$ together with (\ref{sforces}) we have that
$$
\begin{matrix}
\mathrm{I}: & A_0 & A_1 & A_2 & \cdots & & A_{\omega+1} & \cdots\\
\mathrm{II}: & \hspace{36pt} r_0 & \hspace{36pt} r_1 & \hspace{36pt} r_2 & \cdots & r_\omega & \hspace{36pt} r_{\omega+1} & \cdots
\end{matrix}
$$
(where $r_\gamma$ denotes $\sigma(\gamma, A_\gamma)$ for each $\gamma<\omega_1$) forms a play of $G^*(\mathbb{P})$ where Player $\mathrm{II}$ plays according to $\sigma$.
To prove this, the only nontrivial part is the $\subseteq$-continuity of $\seq{A_\gamma\mid\gamma<\omega_1}$. This can be worked out as follows. Work in $V$. For each element of $[\mathbb{P}]^{\leq\aleph_0}_+$ fix its enumeration (allowing overlaps) of order type $\omega$. For each limit $\gamma<\omega_1$ and $n<\omega$ let $D_{\gamma, n}$ be the dense subset of $\mathbb{S}$ consisting of the conditions which decide the least ordinal $\xi<\gamma$ such that the $n$-th element of $\dot{A}_\gamma$ belongs to $\dot{A}_\xi$. Then we may put these dense subsets in the family to which $\mathrm{PFA}$ is applied. This assures that for each limit $\gamma<\omega_1$ every element of $A_\gamma$ belongs to $A_\xi$ for some $\xi<\gamma$.

We also have that, for each $\gamma<\omega_1$ there exists some $\xi_\gamma<\omega_1$ such that $p_\gamma=r_{\xi_\gamma}$ holds. Therefore $p_\gamma$'s have a common extension $p$ in $\mathbb{P}$. By (\ref{pforces}) and the directness of $s_\gamma$'s we have:
$$
p\force_{\mathbb{P}}\text{\lq\lq$\{\dot{q}_\gamma\mid\gamma<\omega_1\}$ is directed $\land\forall\gamma<\omega_1[\dot{q}_\gamma\in\dot{D}_\gamma]$.\rq\rq}
$$
Note that we can pick such $p$ below any given condition of $\mathbb{P}$. This suffices for our conclusion.
\qed(Theorem \ref{thm:PFApreserved})

As an immediate corollary of Theorem \ref{thm:PFApreserved} together with Theorem \ref{thm:wsq}, we can reproduce a proof of the following well-known theorem first proved by Magidor \cite{magidor:_jerusalem}\footnote{As for a written proof of this theorem, it is announced in a recent paper by Cummings and Magidor \cite{cummings:_martinsmaximum} which argues the weak square principles derived from the Martin's Maximum, that it will be dealt with in a further publication by Magidor.}.

\begin{thm}[Magidor]\normalfont
The statement that $\square_{\kappa, \omega_2}$ holds for every cardinal $\kappa\geq\omega_2$ is relatively consistent with $\mathrm{ZFC+PFA}$.
\end{thm}

\begin{proof}
We may assume that in our ground model it holds that $\mathrm{ZFC+PFA}+\text{\lq\lq$2^\kappa=\kappa^+$ for every cardinal $\kappa\geq\omega_2$\rq\rq}$, since the last statement is relatively consistent to $\mathrm{ZFC+PFA}$ (see \cite[Proof of Thereom 17]{yoshinobu13:_oper}). Let
$$\seq{\mathbb{P}_\alpha, \dot{\mathbb{Q}}_\alpha\mid\alpha\in\mathrm{Ord}}$$
be the proper class iterated forcing construction with the Easton support such that $\mathbb{P}_0$ is trivial and that $\force_{\mathbb{P}_\alpha}\text{\lq\lq$\dot{\mathbb{Q}}_\alpha=\mathbb{P}_{\square_{\omega_{2+\alpha}, \omega_2}}$\rq\rq}$ for every ordinal $\alpha$, and let $\mathbb{P}_\infty$ be its direct limit. By standard arguments we have that forcing with $\mathbb{P}_\infty$ preserves $\mathrm{ZFC}$ and cofinalities (for iterations with the Easton support see \cite{baumgartner:_iterated}; for treatment of proper class forcing consult \cite{sdfriedman:_class}), and therefore forces $\square_{\kappa, \omega_2}$ for every cardinal $\kappa\geq\omega_2$. Now  by Theorem \ref{thm:wsq} and Lemma \ref{lma:iteration} $\mathbb{P}_\infty$ is $*$-tactically closed, and thus by Theorem \ref{thm:PFApreserved} $\mathrm{PFA}$ holds in this extension.
\end{proof}

\section{The Setwise Climbability Properties}\label{sec:SCP}

It has been observed that Jensen's square principles and some of their variations can be characterized as a Martin-type axiom for a suitable class of posets. For example, Velleman \cite{velleman_strategy}, and Ishiu and the author \cite{ishiuyoshinobu02} observed that $\square_{\omega_1}$ is equivalent to $\mathrm{MA}_{\omega_2}$ for the class of $(\omega_1+1)$-strategically closed posets. For another example, in \cite{yoshinobu13:_oper} the author introduced the following fragment of $\square_{\omega_1}$ and observed that it is equivalent to $\mathrm{MA}_{\omega_2}$ for the class of operationally closed posets.
 \begin{dfn}\normalfont
$\mathrm{CP}_{\omega_1}$ (the {\it climbability property\/}) is the following statement: There exists a function $f:\omega_2\to\omega_1$ such that for each $\beta\in S^2_1$, there exists a club subset $C$ of $\beta$ with $\mathrm{o.t.}C=\omega_1$ such that $f(\alpha)=\mathrm{o.t.}(C\cap\alpha)$ holds for every $\alpha\in C$.
\end{dfn}

In this section we introduce two more combinatorial principles named as the {\it setwise climbability properties\/}, and show that they are equivalent to $\mathrm{MA}_{\omega_2}$ for the class of $*$-tactically closed posets and that of $*$-operationally closed posets respectively.

\begin{dfn}\normalfont
\begin{enumerate}[(1)]
  \item $\mathrm{SCP}$ is the following statement: There exists a sequence $\seq{z_\alpha\mid\alpha\in S^2_0}$ and a function $f:\omega_2\to\omega_1$ satisfying:
  \begin{enumerate}[(a)]
    \item For each $\alpha\in S^2_0$, $z_\alpha$ is a countable cofinal subset of $\alpha$.
    \item For each $\beta\in S^2_1$, there exists a club subset $C$ of $\beta\cap S^2_0$ with $\mathrm{o.t.}C=\omega_1$ satisfying:
    \begin{enumerate}[(i)]
      \item $\seq{z_\alpha\mid\alpha\in C}$ is increasing and continuous with respect to inclusion.
      \item For each $\alpha\in C$, $f(\alpha)=\mathrm{o.t.}(C\cap\alpha)$ holds.
    \end{enumerate}
  \end{enumerate}
  
  \item $\mathrm{SCP}^-$ is the statement obtained by removing all references to the function $f$ in the above statement of $\mathrm{SCP}$.
\end{enumerate}
\end{dfn}

Now we introduce natural posets respectively for $\mathrm{SCP}$ and $\mathrm{SCP}^-$.

\begin{dfn}\normalfont
We define posets $\mathbb{P}_{\mathrm{SCP}}$ and $\mathbb{P}_{\mathrm{SCP}^-}$ as follows:
\begin{enumerate}[(1)]
\item\label{pscp} A condition $p$ of $\mathbb{P}_{\mathrm{SCP}}$ is of the form
$$
p=\seq{\seq{z^p_\alpha\mid\alpha\in S^2_0\land\alpha\leq\beta^p}, f^p}
$$
satisfying
\begin{enumerate}[(a)]
\item\label{pscpa} $\beta^p$ is an ordinal in $S^2_0$.
\item $f^p:\beta^p+1\to\omega_1$.
\item\label{pscpc} For each $\alpha\in S^2_0$ with $\alpha\leq\beta^p$, $z^p_\alpha$ is a countable cofinal subset of $\alpha$.
\item\label{pscpd} For each $\beta\in \beta^p\cap S^2_1$, there exists a club subset $C$ of $\beta^p\cap S^2_0$ with $\mathrm{o.t.}C=\omega_1$ satisfying:
\begin{enumerate}[(i)]
\item $\seq{z^p_\alpha\mid\alpha\in C}$ is increasing and continuous with respect to inclusion.
\item For each $\alpha\in C$, $f(\alpha)=\mathrm{o.t.}(C\cap\alpha)$.
\end{enumerate}
\end{enumerate}
For $p$, $q\in\mathbb{P}_{\mathrm{SCP}}$, we let $p\leq_{\mathbb{P}_{\mathrm{SCP}}}q$ if $\vec{z}^q=\vec{z}^p\restrict(\beta^q+1)$ and $f^q=q^p\restrict(\beta^q+1)$.
\item A condition $p$ of $\mathbb{P}_{\mathrm{SCP}^-}$ is of the form
$$
p=\seq{z^p_\alpha\mid\alpha\in S^2_0\land\alpha\leq\beta^p}
$$
satisfying (\ref{pscpa}), (\ref{pscpc}) and (\ref{pscpd})(i) in (\ref{pscp}) above.

Both $\mathbb{P}_{\mathrm{SCP}}$ and $\mathbb{P}_{\mathrm{SCP}^-}$ are ordered by initial segment.
\end{enumerate}
\end{dfn}

\begin{lma}\normalfont\label{PSCPlma}
\begin{enumerate}[(1)]
  \item $\mathbb{P}_{\mathrm{SCP}}$ is $*$-operationally closed.\label{PSCP}
  \item $\mathbb{P}_{\mathrm{SCP}^-}$ is $*$-tactically closed.\label{PSCP-}
\end{enumerate}
\end{lma}

\proof We will only show (\ref{PSCP}). (\ref{PSCP-}) is easier. We will define a $*$-operation $\tau:\omega_1\times[\mathbb{P}_{\mathrm{SCP}}]^{\leq\omega}_+\to\mathbb{P}_{\mathrm{SCP}}$. Let $\delta<\omega$ and $A\in[\mathbb{P}_{\mathrm{SCP}}]^{\leq\omega}_+$. If $\delta$ is $0$ or a successor ordinal, set $\tau(\delta, A)$ so that it properly extends the boolean infimum of $A$. This assures that, whenever Player $\mathrm{II}$ plays according to $\tau$, for each limit $\eta<\omega_1$, the union of Player $\mathrm{I}$'s moves made before the $\eta$-th turn of Player $\mathrm{II}$ has no strongest condition. So for the case $\delta$ is a nonzero limit ordinal, we may define $\tau(\delta, A)$ only for $A$ with no strongest condition. For such $A$, there exist $\gamma\in S^2_0$, $\vec{z}=\seq{z_\alpha: \alpha\in S^2_0, \alpha<\gamma}$ and a function $f:\gamma\to\omega_1$ such that the following holds:
$$
A=\{\seq{\vec{z}\restrict(\beta^p+1), f\restrict(\beta^p+1)}\mid p\in A\}.
$$
Then we set
$$
\tau(\delta, A):=\seq{\concat{\vec{z}}{\seq{z_\gamma}}, f\cup\{\seq{\gamma, \overline{\delta}}\}},
$$
where $z_\gamma=\{\beta^p\mid p\in A\}$ and $\overline{\delta}$ is such that $\delta=\omega(1+\overline{\delta})$.

We will show that $\tau$ is a winning $*$-operation. Consider any play of $G^*(\mathbb{P}_{\mathrm{SCP}})$ where Player $\mathrm{II}$ plays according to $\tau$. It is easy to see that $\mathbb{P}_{\mathrm{SCP}}$ is $\omega_1$-closed, and thus it is enough to show that Player $\mathrm{II}$ can make her $\omega_1$-th move. For $\delta<\omega_1$, let $A_\delta$ denote the union of Player $\mathrm{I}$'s moves made before the $\delta$-th move of Player $\mathrm{II}$. Then there exist an ordinal $\gamma\in S^2_1$, $\vec{z}=\seq{z_\alpha: \alpha\in S^2_0, \alpha<\gamma}$ and a function $f:\gamma\to\omega_1$ such that, for each $\delta<\omega_1$ the following holds:
$$
A_\delta=\{\seq{\vec{z}\restrict(\beta^p+1), f\restrict(\beta^p+1)}\mid p\in A_\delta\}.
$$
Let $\gamma_\xi:=\sup\{\beta^p\mid p\in A_{\omega(1+\xi)}\}$ for $\xi<\omega_1$. Then by the definition of $\tau$, $C=\{\gamma_\xi\mid\xi<\omega_1\}$ is a club subset of $\gamma$. Moreover, for each $\xi<\omega_1$ we have $z_{\gamma_\xi}=\{\beta^p\mid p\in A_{\omega(1+\xi)}\}$, and thus $\seq{z_{\gamma_\xi}\mid\xi<\omega_1}$ is increasing and continuous with respect to inclusion. Furthermore, for each $\xi<\omega_1$ we have $f(\gamma_\xi)=\xi=\mathrm{o.t.}(C\cap\gamma_\xi)$.  These facts assure that $\seq{\vec{z}, f}$ can be extended to a condition of $\mathbb{P}_{\mathrm{SCP}}$, which is a common extension of all moves of Player $\mathrm{I}$. This shows that Player $\mathrm{II}$ wins the game.\qed

In particular, both $\mathbb{P}_{\mathrm{SCP}}$ and $\mathbb{P}_{\mathrm{SCP}^-}$ are $(\omega_1+1)$-strategically closed, and therefore preserve cardinalities below $\omega_2$. Moreover, a simple induction argument using this closedness property gives the following density lemma:
\begin{lma}\normalfont\label{densitylemma}
For each $\beta<\omega_2$, $D_\beta=\{p\in\mathbb{P}_{\mathrm{SCP}}\mid \beta^p>\beta\}$ and $D^-_\beta=\{p\in\mathbb{P}_{\mathrm{SCP}^-}\mid \beta^p>\beta\}$ are dense respectively in $\mathbb{P}_{\mathrm{SCP}}$ and $\mathbb{P}_{\mathrm{SCP}^-}$.\qed
\end{lma}
These facts assures that $\mathbb{P}_{\mathrm{SCP}}$ forces $\mathrm{SCP}$ and that $\mathbb{P}_{\mathrm{SCP}^-}$ forces $\mathrm{SCP}^-$.

\begin{thm}\normalfont
\begin{enumerate}[(1)]
  \item The following are equivalent:\label{SCPequiv}
   \begin{enumerate}[(a)]
    \item $\mathrm{SCP}$.\label{SCPa}
    \item Every $*$-operationally closed poset is $\omega_2$-strategically closed.\label{SCPb}
    \item $\mathbb{P}_{\mathrm{SCP}}$ is $\omega_2$-strategically closed.\label{SCPc}
    \item $\mathrm{MA}_{\omega_2}(\text{$*$-operationally closed})$.\label{SCPd}
    \item $\mathrm{MA}_{\omega_2}(\mathbb{P}_{\mathrm{SCP}})$.\label{SCPe}
  \end{enumerate}
  \item\label{SCP-equiv} The statement obtained by replacing each occurrence of `$\mathrm{SCP}$' and `$*$-operationally closed' in (\ref{SCPequiv}) by `$\mathrm{SCP}^-$' and `$*$-tactically closed' respectively is also valid.
\end{enumerate}
\end{thm}

\proof Again we only show (\ref{SCPequiv}), since (\ref{SCP-equiv}) is easier. Since $\mathrm{MA}_{\omega_2}$ is valid for every $\omega_2$-strategically closed poset, (\ref{SCPb}) implies (\ref{SCPd}) and (\ref{SCPc}) implies (\ref{SCPe}) respectively. (\ref{SCPb}) implies (\ref{SCPc}) and (\ref{SCPd}) implies (\ref{SCPe}) by Lemma \ref{PSCPlma}(\ref{PSCP}). 
(\ref{SCPe}) implies (\ref{SCPa}) by Lemma \ref{densitylemma}, since a filter which intersects every $D_\beta$ generates a witness for $\mathrm{SCP}$. Now assume (\ref{SCPa}) and show (\ref{SCPb}) holds. Suppose $\seq{z_\alpha\mid\alpha\in S^2_0}$ and $f$ witness $\mathrm{SCP}$. Let $\mathbb{P}$ be any $*$-operationally closed poset, and $\tau$ a winning $*$-operation for $G^*(\mathbb{P})$. We may assume that $\tau(\delta, A)$ is a common extension of $A$ for every $\delta<\omega_1$ and $A\in[\mathbb{P}]^{\leq\omega}_+$. We will describe how Player $\mathrm{II}$ wins $G_{\omega_2}(\mathbb{P})$. We use symbols $a_\gamma$ and $b_\gamma$ to denote the $\gamma$-th move of Player $\mathrm{I}$ and Player $\mathrm{II}$ respectively. We will use the following lemma (for proof see Ishiu and Yoshinobu \cite[Lemma 2.3 and the proof of Theorem 3.3]{ishiuyoshinobu02}).
\begin{lma}[Ishiu and Yoshinobu]\normalfont\label{directive}
There exists a tree $\mathcal{T}=\seq{S, <_{\mathcal{T}}}$ (where $S=\omega_2\setminus\mathrm{Lim}$)
of height $\omega$ such that
\begin{enumerate}[(1)]
\item For every $\beta$, $\gamma\in S$, $\beta<_{\mathcal{T}}\gamma$ implies $\beta<\gamma$.
\item For every $\alpha\in S^2_0$, there exists a cofinal branch $b$ of $\mathcal{T}$ with $\sup b=\alpha$.\qed
\end{enumerate}
\end{lma}

Let $\mathcal{T}$ be a tree as above. In a play of $G_{\omega_2}(\mathbb{P})$, for each $\gamma<\omega_2$, Player $\mathrm{II}$ may choose her $\gamma$-th move in the following way:
$$
b_\gamma=
\begin{cases}
\tau(\mathrm{ht}_{\mathcal{T}}(\gamma), \{a_\xi\mid\xi\leq_{\mathcal{T}}\gamma\})&\text{if $\gamma\in S$,}\\
\tau(f(\gamma), \{a_\xi\mid\xi\in z_\gamma\})&\text{if $\gamma\in S^2_0$,}\\
\text{any common extension of $\{a_\xi\mid\xi<\gamma\}$}&\text{if $\gamma\in S^2_1$.}
\end{cases}
$$
Let us show that this is a winning strategy. Consider any play of $G_{\omega_2}(\mathbb{P})$ where Player $\mathrm{II}$ plays according to this strategy. It is enough to show that, for each $\gamma<\omega_2$,
\begin{enumerate}[(i)]
  \item The $\gamma$-th turn of Player $\mathrm{II}$ exists, that is, the current position of the turn is nonzero.\label{gamma1}
  \item The $\gamma$-th move of Player $\mathrm{II}$ is legitimate, that is, her move extends the current position of the turn.\label{gamma2}
\end{enumerate}

We will show (\ref{gamma1}) and (\ref{gamma2}) by induction on $\gamma$. So by induction hypothesis we may suppose that all moves preceding the $\gamma$-th move of Player $\mathrm{II}$ have been legitimately made.

\smallskip\noindent\underline{Case 1} $\gamma\in S$.

In this case (\ref{gamma1}) is immediate, since the current position is $a_\gamma$. (\ref{gamma2}) follows from the following inequality:
$$
\tau(\mathrm{ht}_{\mathcal{T}}(\gamma), \{a_\xi\mid\xi\leq_{\mathcal{T}}\gamma\})\leq_{\mathcal{B}(\mathbb{P})}\bigwedge\{a_\xi\mid\xi\leq_{\mathcal{T}}\gamma\}=a_\gamma.
$$

\smallskip\noindent\underline{Case 2} $\gamma\in S^2_0$.

By the definition of $\mathcal{T}$, there is a cofinal branch $b=\seq{\xi_n\mid n<\omega}$ of $\mathcal{T}$ such that $\sup b=\gamma$. For each $n<\omega$ it holds that $\xi_n\in S$ and $\mathrm{ht}_{\mathcal{T}(\xi_n)}=n$, and thus we have $b_{\xi_n}=\tau(n, \{a_{\xi_m}\mid m\leq n\})$. Therefore, setting $A_n=\{a_{\xi_m}\mid m\leq n\}$ for each $n<\omega$,
$$
\begin{matrix}
\mathrm{I}: & A_0  & A_1 & A_2 & \cdots\\
\mathrm{II}: & \hspace{36pt} b_{\xi_0} & \hspace{36pt} b_{\xi_1} & \hspace{36pt} b_{\xi_2} & \cdots
\end{matrix}
$$
forms a part of a play of $G^*(\mathbb{P})$ where Player $\mathrm{II}$ plays according to $\tau$. Therefore $\bigwedge\{b_{\xi_n}\mid n<\omega\}$ is nonzero. Since $\{b_{\xi_n}\mid n<\omega\}$ is cofinal in the moves of Player $\mathrm{I}$ made before the $\gamma$-th move of Player $\mathrm{II}$, $\bigwedge\{b_{\xi_n}\mid n<\omega\}$ is equal to the current position of the turn. This shows (\ref{gamma1}). For (\ref{gamma2}), note that we have
$$
\tau(f(\gamma), \{a_\xi\mid\xi\in z_\gamma\})\leq_{\mathcal{B}(\mathbb{P})}\bigwedge\{a_\xi\mid\xi\in z_\gamma\}.
$$
Since $z_\gamma$ is cofinal in $\gamma$, the right hand side of the above inequality is equal to the current position of the $\gamma$-th turn of Player $\mathrm{II}$.

\smallskip\noindent\underline{Case 3} $\gamma\in S^2_1$.

By the definitions of $\seq{z_\alpha\mid\alpha\in S^2_0}$ and $f$ there exists a club subset $C$ of $\gamma\cap S^2_0$ such that $\seq{z_\alpha\mid\alpha\in C}$ is increasing and continuous with respect to the inclusion, and that $f(\alpha)=\mathrm{o.t.}(C\cap\alpha)$ holds for each $\alpha\in C$. Therefore, letting $\seq{\alpha_\xi\mid\xi<\omega_1}$ be the increasing enumeration of $C$, it holds that $b_{\alpha_\xi}=\tau(\xi, \{a_\eta\mid\eta\in z_{\alpha_\xi}\})$. Therefore, setting $A_\xi=\{a_\eta\mid\eta\in z_{\alpha_\xi}\}$ for each $\xi<\omega_1$,
$$
\begin{matrix}
\mathrm{I}: & A_0 & A_1 & A_2 & \cdots & & A_{\omega+1} & \cdots\\
\mathrm{II}: & \hspace{36pt} b_{\alpha_0} & \hspace{36pt} b_{\alpha_1} & \hspace{36pt} b_{\alpha_2} & \cdots & b_{\alpha_\omega} & \hspace{36pt} b_{{\alpha_{\omega+1}}} & \cdots
\end{matrix}
$$
forms a play of $G^*(\mathbb{P})$ where Player $\mathrm{II}$ plays according to $\tau$. This implies that $\bigwedge\{b_{\alpha_\xi}\mid\xi<\omega_1\}$ is nonzero. Since $\{b_{\alpha_\xi}\mid \xi<\omega_1\}$ is cofinal in the moves made before the $\gamma$-th turn of Player $\mathrm{II}$, $\bigwedge\{b_{\alpha_\xi}\mid\xi<\omega_1\}$ is equal to the current position of the turn. This shows (\ref{gamma1}). In this case (\ref{gamma2}) is clear.\qed

\section{Operations versus $*$-tactics (1): preservation under $*$-tactically closed forcing}\label{sec:opstar1}
In this section we show the following theorem.

\begin{thm}\normalfont\label{thm:CCholds}
Assume $\mathrm{MA}^+(\text{$\omega_1$-closed})$. Then for every $*$-tactically closed poset $\mathbb{P}$, it holds that
$$
\force_\mathbb{P}\mathrm{CC},
$$
where $\mathrm{CC}$ denotes Chang's Conjecture.
\end{thm}

Since Chang's Conjecture negates $\mathrm{CP}$ (see \cite{yoshinobu13:_oper} for proof), Theorem \ref{thm:CPfails} implies the following corollary.

\begin{cor}\normalfont\label{thm:CPfails}
Assume $\mathrm{MA}^+(\text{$\omega_1$-closed})$. Then for every $*$-tactically closed poset $\mathbb{P}$, it holds that
$$
\force_\mathbb{P}\neg\mathrm{CP}.
$$
\end{cor}

Note that since $\mathrm{CP}$ can be forced by an $(\omega_1+1)$-operationally closed forcing, Corollary \ref{thm:CPfails} shows that $(\omega_1+1)$-operational closedness does not imply $*$-tactical closedness. Since $\mathrm{SCP}^-$ can be forced by a $*$-tactically closed forcing, Corollary \ref{thm:CPfails} also shows that $\mathrm{SCP}^-$ does not imply $\mathrm{CP}$.

We will give a proof of Theorem \ref{thm:CCholds} below. Our proof is based on and generalizes that of Miyamoto \cite[Theorem 1.3]{miyamoto:_faandcc}, which obtains a model of $\mathrm{CC}$ with some weak fragment of $\square_{\omega_1}$, assuming (an axiom equivalent to) $\mathrm{MA}^+(\text{$\omega_1$-closed})$ in the ground model. Note also that Miyamoto's argument was extracted from Sakai \cite{sakai:_ccandws}, which obtains a model of $\mathrm{CC}$ with $\square_{\omega_1, 2}$, starting from the ground model with a measurable cardinal.

We will use the following equivalent form of $\mathrm{MA}^+(\text{$\omega_1$-closed})$, introduced by Miyamoto \cite{miyamoto:_faandcc}. 

\begin{dfn}\normalfont\label{dfn:fa}
$\mathrm{FA}^*(\text{$\omega_1$-closed})$ denotes the following statement: For any $\omega_1$-closed poset $\mathbb{P}$, any family of dense subsets $\seq{D_i\mid i<\omega_1}$, any regular cardinal $\theta\geq(2^{2^{|\mathrm{TC}(\mathbb{P})|}})^+$, any structure $\mathfrak{A}=\seq{H_\theta, \in, \{\mathbb{P}\}, \ldots}$, any countable $N\prec\mathfrak{A}$ and any $(N, \mathbb{P})$-generic condition $p\in\mathbb{P}$, there exists a directed subset $F$ of $\mathbb{P}$ of size at most $\aleph_1$ satisfying the following:
\begin{enumerate}[(i)]
\item $F\cap D_i\not=\emptyset$ for every $i<\omega_1$.\label{fdi}
\item $q\leq_\mathbb{P} p$ holds for some $q\in F$.\label{qpp}
\item $N\prec_{\omega_1}N(F)=\{g(F)\mid g:\mathcal{P}(\mathbb{P})\to H_\theta\land g\in N\}$.\label{nnf}
\end{enumerate}
\end{dfn}
\begin{lma}[Miyamoto and Usuba]\normalfont\label{lma:fastar}
$\mathrm{FA}^*(\text{$\omega_1$-closed})$ is equivalent to $\mathrm{MA}^+(\text{$\omega_1$-closed})$.
\end{lma}
See Usuba \cite{usuba:_notesonmiyamotofa} for a proof of Lemma \ref{lma:fastar}.

\begin{lma}\normalfont\label{lma:FAstrong}
Assume $\mathrm{FA}^*(\text{$\omega_1$-closed})$. Let $\mathbb{P}$, $\seq{D_i\mid i<\omega_1}$, $\theta$, $\mathfrak{A}$, $N$ and $p$ are as in Definition \ref{dfn:fa}. If $\mathbb{P}$ collapses $\omega_2$, then $F$ in the conclusion of Definition \ref{dfn:fa} can be taken so that $N\prec^s_{\omega_1}N(F)$ holds (Recall that the notation $\prec^s_\delta$ for $\delta<\omega_2$ was introduced two paragraphs before Definition \ref{dfn:strgen}).
\end{lma}

\proof Pick a $\mathbb{P}$-name $\dot{f}\in N$ for a map from $\omega_1$ onto $\omega^V_2$, and choose $p'\leq_\mathbb{P}p$ so that there exists $\alpha<\omega_2$ with $\alpha\geq\sup(N\cap\omega_2)$ and $i<\omega_1$ such that $p'\force_{\mathbb{P}}\text{\lq\lq$\dot{f}(\check{i})=\check{\alpha}$.\rq\rq}$ Now apply $\mathrm{FA}^*(\text{$\omega_1$-closed})$ for $p'$ instead of $p$ to obtain $F$. Then $F$ also satisfies (\ref{qpp}) for $p$. Note that
$$
\beta:=\sup\{\eta<\omega_2\mid\exists q\in F\exists i<\omega_1[q\force_{\mathbb{P}}\text{\lq\lq$\dot{f}(\check{i})=\check{\eta}$\rq\rq}]\}
$$
satisfies $\sup(N\cap\omega_2)\leq\alpha\leq\beta<\omega_2$ and $\beta\in N(F)$, and therefore we have $N\prec^s_{\omega_1}N(F)$.\qed

\smallskip
We will also use the following auxiliary poset which is induced from a given $*$-tactically closed poset.

\begin{dfn}\normalfont\label{dfn:posetr}
Let $\mathbb{P}$ be a $*$-tactically closed poset, and fix a winning $*$-tactic $\sigma$ for $G^*(\mathbb{P})$.
We define a poset $\mathbb{R}=\mathbb{R}(\mathbb{P}, \sigma)$ as follows: A condition $P$ of $\mathbb{R}$ is either the empty sequence $1_\mathbb{R}=\seq{}$, or of the form $P=\seq{A^P_\xi\mid\xi\leq\alpha^P}$ ($\alpha^P<\omega_1$) satisfying the following:
\begin{enumerate}[(a)]
\item $P$ is a $\subseteq$-continuous increasing sequence of elements of $[\mathbb{P}]^{\leq\omega}_+$.
\item $\sigma(A^P_\xi)\geq_{\mathcal{B}(\mathbb{P})}\bigwedge A^P_{\xi+1}$ for every $\xi<\alpha^P$.
\end{enumerate}
$\mathbb{R}$ is ordered by initial segment.
Note that $P=\seq{A^P_\xi\mid\xi\leq\alpha^P}$ is an $\mathbb{R}$-condition if and only if
$$
\begin{matrix}
\mathrm{I}: & A^P_0 & A^P_1 & \cdots &\\
\mathrm{II}: & \hspace{36pt} \sigma(A^P_0) & \hspace{36pt} \sigma(A^P_1) & \cdots & \sigma(A^P_{\alpha^P})
\end{matrix}
$$
forms a part of a play of $G^*(\mathbb{P})$ where Player $\mathrm{II}$ plays according to $\sigma$. Therefore it is clear that $\mathbb{R}$ is $\omega_1$-closed. In fact, whenever $\seq{P_n\mid n<\omega}$ is a strictly descending sequence of $\mathbb{R}$-conditions, $\bigcup\{P_n\mid n<\omega\}$ is of the form $\seq{A_\xi\mid\xi<\gamma}$ for some limit $\gamma<\omega_1$, and letting $A_\gamma=\bigcup\{A_\xi\mid\xi<\gamma\}$ we have $A_\gamma\in[\mathbb{P}]^{\leq\omega}_+$ because $\sigma$ is a winning $*$-tactic, and $P_\omega=\seq{A_\xi\mid\xi\leq\gamma}$ is the greatest common extension of $\{P_n\mid n<\omega\}$. For $P=\seq{A^P_\xi\mid\xi\leq\alpha^P}\in\mathbb{R}$ we denote $A^P_{\alpha^P}$ simply as $A^P$. We define $\pi:\mathbb{R}\to\mathbb{P}$ by $\pi(1_\mathbb{R})=1_\mathbb{P}$ and $\pi(P)=\sigma(A^P)$ for other $P\in\mathbb{R}$. It is easy to check that $\pi:\mathbb{R}\to\mathbb{P}$ is a projection. We also denote $\alpha^P$ as $\mathrm{lh}(P)$.
\end{dfn}

\begin{lma}\normalfont\label{density}
Let $\mathbb{P}$, $\sigma$ and $\mathbb{R}$ be as above. Then the following subsets of $\mathbb{R}$ are dense and open in $\mathbb{R}$.
\begin{enumerate}[(1)]
  \item $D^\mathrm{lh}_\alpha=\{P\in\mathbb{R}\mid \mathrm{lh}(P)\geq\alpha\}$ for $\alpha<\omega_1$.\label{density:length}
  \item $E_a=\{P\in\mathbb{R}\mid a\in A^P\lor a\perp_{\mathbb{P}}\pi(P)\}$ for $a\in\mathbb{P}$.\label{density:elements}
 \end{enumerate}
\end{lma}

\proof Straightforward.\qed

\begin{lma}\normalfont\label{collapse}
Let  $\mathbb{P}$, $\sigma$ and $\mathbb{R}$ be as in Definition \ref{dfn:posetr}. If $\mathbb{P}$ is non-atomic, then $\mathbb{R}$ collapses $\omega_2$.
\end{lma}

\proof Note first that $\mathbb{R}$ is a tree of height $\omega_1$. Since $\mathbb{P}$ is non-atomic and $(\omega_1+1)$-strategically closed, each condition of $\mathbb{P}$ has at least $2^{\aleph_1}\geq\aleph_2$ pairwise incompatible extensions. This implies that each node of $\mathbb{R}$ has at least $\aleph_2$ distinct immediate successors. Using this fact one can let each node of $\mathbb{R}$ code a function from a countable ordinal to $\omega_2$, so that each $\mathbb{R}$-generic filter codes a function from $\omega_1$ to $\omega^V_2$, which is surjective by genericity .\qed

\proof[Proof of Theorem \ref{thm:CCholds}] Since $\mathrm{MA}^+(\text{$\omega_1$-closed})$ implies $\mathrm{CC}$, as shown in Foreman, Magidor and Shelah \cite{foreman88}, we may assume that $\mathbb{P}$ is non-atomic. Let $\sigma$ and $\mathbb{R}$ be as in Definition \ref{dfn:posetr}. Let $\theta=(2^{2^{|\mathrm{TC}(\mathbb{P})|}})^+$ (and thus $\sigma$, $\mathbb{R}\in H_\theta$). Let $\dot{F}\in H_\theta$ be any $\mathbb{P}$-name for a function $[\omega_2]^{<\omega}\to\omega_2$ and let $\mathfrak{A}=\seq{H_\theta, \in, \{\mathbb{P}, \sigma, \mathbb{R}, \dot{F}\}}$. It is enough to show that, for any $p\in\mathbb{P}$, there exist $q\leq_\mathbb{P}p$ and $N\prec\mathfrak{A}$ such that $N\cap\omega_1<\omega_1$, $|N\cap\omega_2|=\aleph_1$ and that $q$ is $(N, \mathbb{P})$-generic.

By recursion we will simultaneously define $p_\gamma\in\mathbb{P}$ for $\gamma\in\omega_1\setminus\mathrm{Lim}$, a countable elementary substructure $N_\gamma$ of $\mathfrak{A}$, $A_\gamma\in[\mathbb{P}]^{\leq\omega}_+$ and $P_\gamma\in\mathbb{R}$ respectively for $\gamma<\omega_1$ so that the following requirements are satisfied: For each $\xi<\omega_1$,
\begin{enumerate}[(i)]
\item $\sigma(A_\zeta)\geq_\mathbb{P}p_\xi$ if $\xi=\zeta+1$.\label{Pgeqp}
\item $N_\zeta\prec^s_{\omega_1}N_\xi$ if $\xi=\zeta+1$.\label{Ninc}
\item $N_\xi=\bigcup\{N_\zeta\mid\zeta<\xi\}$ if $\xi\in\mathrm{Lim}$.\label{Nconti}
\item $p_\xi\in A_\xi$ if $\xi\notin\mathrm{Lim}$.\label{pinA}
\item $A_\zeta\subseteq A_\xi$ if $\xi=\zeta+1$.\label{Ainc}
\item $A_\xi=\bigcup\{A_\zeta\mid\zeta<\xi\}$ if $\xi\in\mathrm{Lim}$.\label{Aconti}
\item $P_\xi$ is the greatest common extension of an $(N_\xi, \mathbb{R})$-generic sequence (and thus is $(N_\xi, \mathbb{R})$-strongly generic).\label{Pgen}
\item $A_\xi=A^{P_\xi}$.\label{AAP}
\item $A_\xi\subseteq N_\xi$.\label{AcontdN}
\end{enumerate}

Let $\gamma<\omega_1$ and suppose $\seq{p_\xi\mid\xi\in\gamma\setminus\mathrm{Lim}}$, $\seq{N_\xi\mid\xi<\gamma}$, $\seq{A_\xi\mid\xi<\gamma}$ and $\seq{P_\xi\mid\xi<\gamma}$ are already defined and satisfy (i)--(ix) for every $\xi<\gamma$. We will define $p_\gamma$ (if $\gamma\notin\mathrm{Lim}$), $N_\gamma$, $A_\gamma$ and $P_\gamma$ so that (i)--(ix) for $\xi=\gamma$ hold.

\smallskip
\noindent\underline{Case 1} $\gamma=0$.

First let $p_0=p$, and pick a countable $N_0\prec\mathfrak{A}$ such that $p_0\in N_0$. Let $\delta=N_0\cap\omega_1$. Now pick an $(N_0, \mathbb{R})$-generic sequence $\seq{P_{0, n}\mid n<\omega}$ such that $P_{0, 0}=\seq{\{p_0\}}$. Now let $P_0$ be the greatest common extension of $\{P_{0, n}\mid n<\omega\}$ and let $A_0:=A^{P_0}$. Then it is easy to see (\ref{pinA}), (\ref{Pgen}), (\ref{AAP}) and (\ref{AcontdN}) for $\xi=0$.

\smallskip
\noindent\underline{Case 2} $\gamma=\zeta+1$.

Since $P_\zeta$ is $(N_\zeta, \mathbb{R})$-generic by (\ref{Pgen}) for $\xi=\zeta$, we can apply Lemma \ref{lma:FAstrong} to $\mathbb{R}$, dense subsets $D^\mathrm{lh}_\alpha$ ($\alpha<\omega_1$), $N_\zeta$ and $P_\zeta$ to get a directed subset $F$ of $\mathbb{R}$ satisfying:
\begin{enumerate}[(1)]
\item $F\cap D^{\mathrm{lh}}_\alpha\not=\emptyset$ for every $\alpha<\omega_1$.\label{FDi}
\item $Q\in F$ for some $Q\leq_\mathbb{R}P_\zeta$.\label{QR}
\item $N_\zeta\prec^s_{\omega_1}N_\zeta(F)$.\label{NprecN}
\end{enumerate}
Now let $N_\gamma:=N_\zeta(F)$. By (\ref{NprecN}) we have (\ref{Ninc}) for $\xi=\gamma$. By (\ref{FDi}) and the directedness of $F$, $\bigcup F$ is of the form $\seq{A^F_i\mid i<\omega_1}$. Note that
$$
\begin{matrix}
\mathrm{I}: & A^F_0 & A^F_1 & \cdots & & A^F_{\omega+1} & \cdots\phantom{,}\\
\mathrm{II}: & \hspace{36pt} \sigma(A^F_0) & \hspace{36pt}  \sigma(A^F_1) & \cdots & \sigma(A^F_\omega) & \hspace{36pt} \sigma(A^F_{\omega+1}) & \cdots
\end{matrix}
$$
forms a play of $G^*(\mathbb{P})$ where Player $\mathrm{II}$ plays according to $\sigma$. Since $\sigma$ is a winning $*$-tactic, $\bigcup\{A^F_i\mid i<\omega_1\}$ has a common extension in $\mathbb{P}$. Since $F\in N_\gamma$, we can pick such a common extension $p_\gamma$ within $N_\gamma$. By (\ref{QR}) $P_\zeta$ is an initial segment of $\bigcup F$. By (\ref{Ninc}) and (\ref{Nconti}) for $\xi\leq\zeta$ we have $N_\zeta\cap\omega_1=\delta$, and thus by Lemma \ref{density}(\ref{density:length}) and (\ref{Pgen}) for $\xi=\zeta$ we have $\mathrm{lh}(P_\zeta)=\delta$ and thus $A^{P_\zeta}=A^F_\delta$ holds. Therefore by (\ref{AAP}) for $\xi=\zeta$ it holds that $\sigma(A_\zeta)=\sigma(A^{P_\zeta})=\sigma(A^F_\delta)\geq_\mathbb{\mathcal{B}(P)}\bigwedge A^F_{\delta+1}\geq_\mathbb{\mathcal{B}(P)}p_\gamma$. This gives (\ref{Pgeqp}) for $\xi=\gamma$.

Now construct $A_\gamma$ and $P_\gamma$ in the same way as $A_0$ and $P_0$: Pick an $(N_\gamma, \mathbb{R})$-generic sequence $\seq{P_{\gamma, n}\mid n<\omega}$ such that $P_{\gamma, 0}=\seq{\{p_\gamma\}}$. Then let $P_\gamma$ be the greatest common extension of $\{P_{\gamma, n}\mid n<\omega\}$ and let $A_\gamma:=A^{P_\gamma}$. Then we have (\ref{pinA}), (\ref{Pgen}), (\ref{AAP}) and (\ref{AcontdN}) for $\xi=\gamma$.

Let $q\in A_\zeta$ be arbitrary. Note that $E_q$ as in Lemma \ref{density}(\ref{density:elements}) is dense open in $\mathbb{R}$ and belongs to $N_\zeta\subseteq N_\gamma$. Thus we have $P_\gamma\in E_q$ by (\ref{Pgen}) for $\xi=\gamma$. So either $q\in A_\gamma$ or $q\perp_{\mathbb{P}}\pi(P_\gamma)$ holds. By (\ref{Pgeqp}) for $\xi=\gamma$, $p_\gamma$ is a common extension of $A_\zeta$ and thus $q\geq_\mathbb{P}p_\gamma$ holds. On the other hand, since $p_\gamma\in A_\gamma$ by (\ref{pinA}) for $\xi=\gamma$ and $\pi(P_\gamma)=\sigma(A_\gamma)$ is a common extension of $A_\gamma$, we have $p_\gamma\geq_{\mathbb{P}}\pi(P_\gamma)$. Thus $q\geq_{\mathbb{P}}\pi(P_\gamma)$ holds and in particular they are compatible. Therefore $q\in A_\gamma$ holds. This shows $A_\zeta\subseteq A_\gamma$, that is (\ref{Ainc}) for $\xi=\gamma$.

\smallskip
\noindent\underline{Case 3} $\gamma\in\mathrm{Lim}$.

Let $N_\gamma:=\bigcup\{N_\zeta\mid\zeta<\gamma\}$ and $A_\gamma:=\bigcup\{A_\zeta\mid\zeta<\gamma\}$. Then by induction hypothesis we immediately have (\ref{Nconti}), (\ref{Aconti}) and (\ref{AcontdN}) for $\xi=\gamma$. So it is enough to find a $P_\gamma$ satisfying (\ref{Pgen}) and (\ref{AAP}) for $\xi=\gamma$. Let $\seq{D_n\mid n<\omega}$ be an enumeration of the dense subsets of $\mathbb{R}$ in $N_\gamma$. Fix a strictly increasing sequence $\seq{\gamma_n\mid n<\omega}$ of ordinals converging to $\gamma$ so that $D_n\in N_{\gamma_n}$ holds for every $n<\omega$.

By induction we construct a descending sequence $\seq{P_{\gamma, n}\mid n<\omega}$ in $\mathbb{R}$ satisfying the following requirements for each $n<\omega$:
\begin{itemize}
\item $P_{\gamma, n}\in N_{\gamma_n}$.
\item $P_{\gamma, n+1}\in D_n$.
\item $\pi(P_{\gamma, n})\geq_{\mathbb{P}}\pi(P_{\gamma_n})$.
\end{itemize}
First let $P_{\gamma, 0}:=1_\mathbb{R}$. It is clear that $P_{\gamma, 0}\in N_{\gamma_0}$ and $\pi(P_{\gamma, 0})=1_\mathbb{P}\geq_{\mathbb{P}}\pi(P_{\gamma_0})$. Suppose $P_{\gamma, n}\in N_{\gamma_n}$ was defined and satisfies $\pi(P_{\gamma, n})\geq_{\mathbb{P}}\pi(P_{\gamma_n})$. Note that since (\ref{Pgen}) for $\xi=\gamma_n$ holds and $\pi$ is a projection, $\pi(P_{\gamma_n})$ is $(N_{\gamma_n}, \mathbb{P})$-strongly generic. Thus since $D_n\in N_{\gamma_n}$, by Lemma \ref{lma:onestep} we can pick $P_{\gamma, n+1}\leq_{\mathbb{R}}P_{\gamma, n}$ such that $P_{\gamma, n+1}\in D_n\cap N_{\gamma_n}$ and that $\pi(P_{\gamma, n+1})\geq_{\mathbb{P}}\pi(P_{\gamma_n})$. But by (\ref{Ninc}) and (\ref{Nconti}) for $\xi<\gamma$ we have $N_{\gamma_n}\subseteq N_{\gamma_{n+1}}$, and by (\ref{Pgeqp}), (\ref{pinA}), (\ref{Ainc}), (\ref{Aconti}) and (\ref{AAP}) for $\xi<\gamma$ we have
\begin{eqnarray*}
\pi(P_{\gamma_n})&=&\sigma(A_{\gamma_n})\geq_{\mathbb{P}}p_{\gamma_n+1}\geq_{\mathcal{B}(\mathbb{P})} \bigwedge A_{\gamma_n+1}\\
&\geq_{\mathcal{B}(\mathbb{P})}&\bigwedge A_{\gamma_{n+1}}\geq_{\mathcal{B}(\mathbb{P})}\sigma(A_{\gamma_{n+1}})=\pi(P_{\gamma_{n+1}}).
\end{eqnarray*}
Therefore we have $P_{\gamma, n+1}\in N_{\gamma_{n+1}}$ and $\pi(P_{\gamma, n+1})\geq_{\mathbb{P}}\pi(P_{\gamma_{n+1}})$, which finishes the construction.

Note that $\seq{P_{\gamma, n}\mid n<\omega}$ is an $(N_\gamma, \mathbb{R})$-generic sequence. Now let $P_\gamma$ be the greatest common extension of $\{P_{\gamma, n}\mid n<\omega\}$. Then $P_\gamma$ is $(N_\gamma, \mathbb{R})$-strongly generic. So it is enough to show the following.

\smallskip
\noindent\underline{Claim} $A^{P_\gamma}=A_\gamma$.
\proof[Proof of Claim] Pick any $q\in A^{P_\gamma}$. Then $q\in A^{P_{\gamma, n}}$ for some $n<\omega$. Since $P_{\gamma, n}\in N_{\gamma_n}$ and $A^{P_{\gamma, n}}$ is countable, $q\in N_{\gamma_n}$. Since $P_{\gamma_n}$ is $(N_{\gamma_n}, \mathbb{R})$-strongly generic, we have $P_{\gamma_n}\in E_q$. Therefore either $q\in A_{\gamma_n}$ or $q\perp_{\mathbb{P}}\pi(P_{\gamma_n})$ holds. But since $q\in A^{P_{\gamma, n}}$ by our construction of $P_{\gamma, n}$ we have
$$
q\geq_{\mathcal{B}(\mathbb{P})}\bigwedge A^{P_{\gamma, n}}\geq_{\mathcal{B}(\mathbb{P})}\pi(P_{\gamma, n})\geq_\mathbb{P}\pi(P_{\gamma_n}),
$$
and therefore $q\in A_{\gamma_n}\subseteq A_\gamma$ must be the case. This shows that $A^{P_\gamma}\subseteq A_\gamma$.

Now pick any $r\in A_\gamma$. Since $A_\gamma\subseteq N_\gamma$, $E_r=D_m$ for some $m<\omega$. By (\ref{Ainc}) and (\ref{Aconti}) for $\xi<\gamma$ we can pick $n>m$ so that $r\in A_{\gamma_n}$ holds. Since $E_r$ is open, $P_{\gamma, m+1}\geq_{\mathbb{R}}P_{\gamma, n}\in E_r$ holds and thus either $r\in A^{P_{\gamma, n}}$ or $r\perp_{\mathbb{P}}\pi(P_{\gamma, n})$ holds. On the other hand, by our construction of $P_{\gamma, n}$ and (\ref{Pgeqp}), (\ref{pinA}), (\ref{Ainc}) and (\ref{AAP}) for $\xi<\gamma$ we have
$$
\pi(P_{\gamma, n})\geq_{\mathbb{P}}\pi(P_{\gamma_n})=\sigma(A_{\gamma_n})\geq_{\mathbb{P}}p_{\gamma_n+1}\in A_{\gamma_n+1}.
$$
Therefore both $r$ and $\pi(P_{\gamma, n})$ are in $A_{\gamma_n+1}$ (by (\ref{Ainc}) for $\xi<\gamma$) and thus are compatible. So $r\in A^{P_{\gamma, n}}\subseteq A^{P_\gamma}$ must be the case. This shows that $A_\gamma\subseteq A^{P_\gamma}$.\qed{(Claim)}

We just have finished our construction of $p_\gamma$'s , $N_\gamma$'s, $A_\gamma$'s and $P_\gamma$'s. Let $N:=\bigcup\{N_\gamma\mid\gamma<\omega_1\}$. By (\ref{Ninc}) and (\ref{Nconti}) we have that $N\cap\omega_1=\delta<\omega_1$ and that $|N\cap\omega_2|=\aleph_1$. By (\ref{Pgeqp}), (\ref{pinA}), (\ref{Ainc}) and (\ref{Aconti}) we have that
$$
\begin{matrix}
\mathrm{I}: & A_0 & A_1 & \cdots & & A_{\omega+1} & \cdots\phantom{,}\\
\mathrm{II}: & \hspace{36pt} \sigma(A_0) & \hspace{36pt}  \sigma(A_1) & \cdots & \sigma(A_\omega) & \hspace{36pt} \sigma(A_{\omega+1}) & \cdots
\end{matrix}
$$
forms a play of $G^*(\mathbb{P})$ where Player $\mathrm{II}$ plays according to $\sigma$. Therefore there exists a common extension $q\in\mathbb{P}$ of $A:=\bigcup\{A_\gamma\mid\gamma<\omega_1\}$. By $p=p_0\in A_0\subseteq A$ we have $q\leq_\mathbb{P}p$. For each $\gamma<\omega_1$, by (\ref{Pgen}) we have $P_\gamma$ is strongly $(N_\gamma, \mathbb{R})$-generic, and since $\pi$ is a projection it follows that $\pi(P_\gamma)=\sigma(A_\gamma)$ is strongly $(N_\gamma, \mathbb{P})$-generic. But $\sigma(A_\gamma)\geq_{\mathcal{B}(\mathbb{P})}\bigwedge A_{\gamma+1}\geq_{\mathcal{B}(\mathbb{P})}q$ holds and thus $q$ is also strongly $(N_\gamma, \mathbb{P})$-generic. Therefore $q$ is $(N, \mathbb{P})$-generic. This completes our proof of Theorem \ref{thm:CCholds}. \qed{(Theorem \ref{thm:CCholds})}

\section{Operations versus $*$-tactics (2): preservation under operationally closed forcing}\label{sec:opstar2}

In this section we show the following theorem, which can be considered as a counterpart of Theorem \ref{thm:CPfails}.

\begin{thm}\normalfont\label{thm:SCPfails}
Assume $\mathrm{MA}^+(\text{$\omega_1$-closed})$. Then for any $(\omega_1+1)$-operationally closed poset $\mathbb{P}$, it holds that
$$
\force_\mathbb{P}\neg\mathrm{SCP}^-.
$$
\end{thm}

Since $\mathrm{SCP}^-$ can be forced by a $*$-tacitically closed poset, Theorem \ref{thm:SCPfails} shows that $*$-tactical closedness does not imply $(\omega_1+1)$-operational closedness. Since $\mathrm{CP}$ can be forced by an $(\omega_1+1)$-operationally closed poset, it also shows that $\mathrm{CP}$ does not imply $\mathrm{SCP}^-$.

Note that, since a poset is $(\omega_1+1)$-operationally closed if and only if its Boolean completion is (see \cite[Lemma 8]{yoshinobu13:_oper}), to prove Theorem \ref{thm:SCPfails} we may assume that $\mathbb{P}=\mathcal{B}\setminus\{0_\mathcal{B}\}$ for a complete Boolean algebra $\mathcal{B}$. So in the rest of this section we fix such $\mathcal{B}$ and $\mathbb{P}$, and we also fix a winning operation $\tau:(\omega_1+1)\times\mathbb{P}\to\mathbb{P}$.

\smallskip
We define an auxiliary poset which plays a similar role as $\mathbb{R}$ in \S 5.

\begin{dfn}\normalfont\label{dfn:S}
We define a poset $\mathbb{S}=\mathbb{S}(\mathbb{P}, \tau)$ as follows: A condition $s$ in $\mathbb{S}$ is either the empty sequence $1_{\mathbb{S}}=\seq{}$, or of the form $\seq{a^s_\gamma\mid\gamma\leq\alpha^s}$ ($\zeta^s<\omega_1$) satisfying
\begin{enumerate}[(a)]
  \item $a^s_{\gamma+1}\leq_\mathbb{P}\tau(\gamma, a^s_\gamma)$ for each $\gamma<\alpha^s$, and
  \item $a^s_\gamma=\bigwedge\{a^s_\xi\mid\xi<\gamma\}$ for each nonzero limit $\gamma\leq\alpha^s$.
\end{enumerate}
$\mathbb{S}$ is ordered by initial segment. Note that $s=\seq{a^s_\gamma\mid\gamma\leq\alpha^s}$ is an $\mathbb{S}$-condition if and only if
$$
\begin{matrix}
\mathrm{I}: & a^s_0 & a^s_1 & \cdots &&a^s_{\omega+1}& \cdots &\\
\mathrm{II}: & \hspace{36pt} b^s_0 & \hspace{36pt} b^s_1 & \cdots & b^s_\omega &\hspace{36pt} b^s_{\omega+1} & \cdots & b^s_{\alpha^s}
\end{matrix}
$$
(where $b^s_\gamma:=\tau(\gamma, a^s_\gamma)$) forms a part of a play of $G_{\omega_1+1}(\mathbb{P})$ where Player $\mathrm{II}$ plays according to $\tau$. Therefore it is clear that $\mathbb{S}$ is $\omega_1$-closed.
For $s=\seq{a^s_\gamma\mid\gamma\leq\alpha^s}\in\mathbb{S}$ we denote $\alpha^s$ and $\tau(\alpha^s, a^s_{\alpha^s})$ as $\mathrm{lh}(s)$ and $\mathrm{lm}(s)$ respectively. We also set $\mathrm{lm}(1_{\mathbb{S}})=1_{\mathbb{P}}$. Then it is easy to see that $\mathrm{lm}:\mathbb{S}\to\mathbb{P}$ is a projection.
\end{dfn}
\begin{lma}\normalfont\label{lma:lengthdensity}
For each $\zeta<\omega_1$, $D^{\mathrm{lh}}_\zeta:=\{s\in\mathbb{S}\mid\mathrm{lh}(s)\geq \zeta\}$ is dense in $\mathbb{S}$.
\end{lma}
\proof Straightforward.\qed

\begin{lma}\normalfont\label{lma:propofS}
$\mathbb{S}$ collapses $\omega_2$.
\end{lma}
\proof Similar to the proof of Lemma \ref{collapse}.\qed

\smallskip
The following is the key lemma to prove Theorem \ref{thm:SCPfails}, though it is provable in $\mathrm{ZFC}$ and does not require $\mathrm{MA}^+(\text{$\omega_1$-closed})$.

\begin{lma}\normalfont\label{lma:stationary}
Suppose $\dot{Z}$ is a $\mathbb{P}$-name such that
\begin{eqnarray}\label{eq:cond}
&\force_\mathbb{P}&\text{\lq\lq $\dot{Z}$ is a function on $S^2_0$ such that for every $\alpha\in S^2_0$}\nonumber\\
&&\text{$\dot{Z}(\alpha)$ is a countable unbounded subset of $\alpha$.\rq\rq}
\end{eqnarray}
Then it holds that
\begin{equation}\label{eq:stat}
\force_\mathbb{S}\text{\lq\lq$\{x\in[\check{({\omega_2}^V)}]^\omega\mid\dot{Z}(\sup x)\nsubseteq x\}$ is stationary in $[\check{({\omega_2}^V)}]^\omega$.\rq\rq}
\end{equation}
\end{lma}

Let us assume Lemma \ref{lma:stationary} for a while and prove Theorem \ref{thm:SCPfails} first.

\proof[Proof of Theorem \ref{thm:SCPfails} {\rm (assuming Lemma \ref{lma:stationary})}]

Assume $\mathrm{MA}^+(\text{$\omega_1$-closed})$ holds in $V$. Suppose for a contradiction that there exists a $\mathbb{P}$-name $\dot{Z}$ satisfying (\ref{eq:cond}) of Lemma \ref{lma:stationary} and $p\in\mathbb{P}$ such that:
\begin{eqnarray}\label{eq:pforce}
p&\force_\mathbb{P}&\text{\lq\lq For every $\beta\in S^2_1$ there exists a club subset $C$ of $\beta\cap S^2_0$}\nonumber\\
&&\text{such that $\mathrm{o.t.} C=\omega_1$ and that $\seq{\dot{Z}(\alpha)\mid\alpha\in C}$ is}\nonumber\\
&&\text{$\subseteq$-continuous increasing.\rq\rq}
\end{eqnarray}
By Lemma \ref{lma:stationary} we have (\ref{eq:stat}). Since $\mathbb{S}$ is $\omega_1$-closed and collapses $\omega_2$, one can choose an $\mathbb{S}$-name $\dot{N}$ such that
\begin{eqnarray}\label{eq:forceS}
&\force_\mathbb{S}&\text{\lq\lq$\dot{N}$ is a function on $\omega_1$ such that $\seq{\dot{N}(\zeta)\mid \zeta<\omega_1}$ is a $\subseteq$-continuous}\nonumber\\
&&\text{increasing sequence of countable sets satisfying}\nonumber\\
&&\text{$\bigcup\{\dot{N}(\zeta)\mid\zeta<\omega_1\}=\check{({\omega_2}^V)}$ and that for each $\zeta<\omega_1$}\nonumber\\
&&\text{it holds that $\sup\dot{N}(\zeta)\notin\dot{N}(\zeta)$, $\sup\dot{N}(\zeta)<\sup\dot{N}(\zeta+1)$ and}\nonumber\\
&&\text{$\dot{Z}(\sup\dot{N}(\zeta))\subseteq \dot{N}(\zeta+1)$.\rq\rq}
\end{eqnarray}
Therefore,  in $V^\mathbb{S}$, $\{\dot{N}(\zeta)\mid \zeta<\omega_1\}$ forms a club subset of $[\check{({\omega_2}^V)}]^{\leq\omega}$ and  thus by (\ref{eq:stat}) we have
$$
\force_\mathbb{S}\text{\lq\lq$\dot{T}:=\{\zeta<\omega_1\mid\dot{Z}(\sup \dot{N}(\zeta))\nsubseteq\dot{N}(\zeta)\}$ is stationary.\rq\rq}
$$
Note that, for each $\zeta<\omega_1$,
$$
D^0_\zeta:=\{s\in\mathbb{S}\mid\text{$s$ decides the values of $\dot{N}(\check{\zeta})$ and $\dot{Z}(\sup\dot{N}(\check{\zeta}))$}\}
$$
is a dense subset of $\mathbb{S}$.

Now apply $\mathrm{MA}^+(\text{$\omega_1$-closed})$ to get a filter $F$ on $\mathbb{S}$ such that $\seq{p}\in F$, $F$ intersects $D^{\mathrm{lh}}_\zeta$ and $D^0_\zeta$ for all $\zeta<\omega_1$, and that
$$
T=\{\zeta<\omega_1\mid\exists s\in F[s\force_\mathbb{S}\text{\lq\lq $\check{\zeta}\in\dot{T}$\rq\rq}]\}
$$
is stationary in $\omega_1$. Note that, since $F$ intersects $D^{\mathrm{lh}}_\zeta$ for all $\zeta<\omega_1$, $\bigcup F$ is of the form $\langle a_\gamma\mid\gamma<\omega_1\rangle$ with $a_0=p$, and
$$
\begin{matrix}
\mathrm{I}: & a_0 & a_1 & \cdots &&a_{\omega+1}& \cdots\\
\mathrm{II}: & \hspace{36pt} b_0 & \hspace{36pt} b_1 & \cdots & b_\omega &\hspace{36pt} b_{\omega+1} & \cdots
\end{matrix}
$$
(where $b_\gamma:=\tau(\gamma, a_\gamma)$) forms a play of $G_{\omega+1}(\mathbb{P})$ where Player $\mathrm{II}$ plays according to $\tau$. Therefore this sequence has a common extension $q\in\mathbb{P}$. Note also that each element of $F$ is an initial segment of $\bigcup F$, and thus $F$ is $\omega_1$-directed. For each $\zeta<\omega_1$ let $N_\zeta$ and $z_\zeta$ be such that
\begin{equation}\label{eq:riforce}
\text{$r_\zeta\force_\mathbb{S}\text{\lq\lq$\dot{N}(\check{\zeta})=\check{N}_\zeta$ and $\dot{Z}(\sup\dot{N}(\check{\zeta}))=\check{z}_\zeta$\rq\rq}$}
\end{equation}
holds for some $r_\zeta\in F$. Such $N_\zeta$ and $z_\zeta$ uniquely exist by the directedness of $F$ and the fact that $F$ intersects $D^0_\zeta$. Moreover, by the definition of $T$, for each $\zeta\in T$ we may assume
\begin{equation}\label{eq:iint}
r_\zeta\force_\mathbb{S}\text{\lq\lq$\dot{Z}(\sup\dot{N}(\check{\zeta}))\nsubseteq\dot{N}(\check{\zeta})$.\rq\rq}
\end{equation}
Let $\alpha_\zeta=\sup N_\zeta$ for $\zeta<\omega_1$. By (\ref{eq:forceS}), (\ref{eq:riforce}), (\ref{eq:iint}) and the $\omega_1$-directedness of $F$ we have the following:
\begin{enumerate}[(a)]
\item $\seq{N_\zeta\mid \zeta<\omega_1}$ forms a $\subseteq$-increasing continuous sequence of countable subsets of $\omega_2$ and $\seq{\alpha_\zeta\mid \zeta<\omega_1}$ is a continuous strictly increasing sequence.
\item For each $\zeta<\omega_1$ it holds that $\sup z_\zeta=\alpha_\zeta$ and $z_\zeta\subseteq N_{\zeta+1}$, and therefore $\bigcup\{z_\xi\mid\xi<\zeta\}\subseteq N_\zeta$ if $\zeta$ is limit.\label{cond:b}
\item For each $\zeta\in T$, $z_\zeta\nsubseteq N_\zeta$ holds.\label{cond:c}
\end{enumerate}

Now for each $\zeta<\omega_1$, by (\ref{eq:riforce}) and the fact that $\sup z_\zeta=\alpha_\zeta$, it holds that
\begin{equation}
r_\zeta\force_\mathbb{S}\text{\lq\lq$\dot{Z}(\check{\alpha_\zeta})=\check{z}_\zeta$.\rq\rq}
\end{equation}
Since $\dot{Z}$ is a $\mathbb{P}$-name and $\mathrm{lm}$ is a projection, by absoluteness we have
$$
\mathrm{lm}(r_\zeta)\force_\mathbb{P}\text{\lq\lq$\dot{Z}(\check{\alpha_\zeta})=\check{z}_\zeta$.\rq\rq}
$$
Since $q$ extends $\mathrm{lm}(r_\zeta)$ for all $\zeta<\omega_1$ we have
\begin{equation}\label{eq:qz}
q\force_\mathbb{P}\text{\lq\lq$\forall \zeta<\omega_1[\dot{Z}(\check{\alpha_\zeta})=\check{z}_\zeta]$.\rq\rq}
\end{equation}
Since $q$ also extends $p$, by (\ref{eq:pforce}) we have
\begin{eqnarray}\label{eq:qclub}
q&\force_\mathbb{P}&\text{\lq\lq There exists a club subset $C$ of $\omega_1$ such that $\seq{\dot{Z}({\check{\alpha}_\zeta})\mid \zeta\in C}$}\nonumber\\
&&\phantom{\lq\lq}\text{is a $\subseteq$-continuous increasing sequence.\rq\rq}
\end{eqnarray}
By (\ref{eq:qz}), (\ref{eq:qclub}) and the fact that $\mathbb{P}$ is $\omega_2$-Baire, there exist $q'\leq_\mathbb{P}q$ and a club subset $C_0$ (in $V$) of $\omega_1$ such that
$$
q'\force_\mathbb{P}\text{\lq\lq$\seq{z_\zeta\mid \zeta\in C_0}\check{}$ is a $\subseteq$-continuous increasing sequence.\rq\rq}
$$
By absoluteness, $\seq{z_\zeta\mid \zeta\in C_0}$ is a $\subseteq$-continuous increasing sequence in $V$.
Now since $T$ is stationary in $\omega_1$, there is an ordinal $\zeta_0\in T\cap\mathrm{l.p.}(C_0)$. Thus on the one hand, by (\ref{cond:c}) we have that $z_{\zeta_0}\nsubseteq N_{\zeta_0}$, and on the other hand, by (\ref{cond:b}) we have
$$
z_{\zeta_0}=\bigcup\{z_\xi\mid\xi\in \zeta_0\cap C_0\}\subseteq\bigcup\{z_\xi\mid\xi<\zeta_0\}\subseteq N_{\zeta_0}.
$$
This is a contradiction and finishes the proof of Theorem \ref{thm:SCPfails}.\qed

\smallskip
Now let us go back to the proof of Lemma \ref{lma:stationary}. To this end, we introduce another type of two player game, which is a variation of the one introduced by Velickovic \cite[p.272]{velickovic92:_forcin}.

\begin{dfn}\normalfont
Let $\theta$ be an uncountable regular cardinal satisfying $\{\mathbb{P}, \tau\}\in H_\theta$, and $\mathfrak{A}$ a model of the form $\seq{H_\theta, \in, \mathbb{P}, \tau, \ldots}$. Let $\mathcal{C}$ denote the set of countable elementary submodels of $\mathfrak{A}$. For $p\in \mathbb{P}$, $\mathfrak{G}(\mathfrak{A}, p)$ denotes the following two-player game. Players choose their moves as follows:

\begin{table}[h]
\hspace*{2cm}
\begin{tabular}{ccccccl}
$\mathrm{I}$ & : & $\seq{N_0, p_0}$ & & $\seq{N_1, p_1}$ & & $\cdots$\\
$\mathrm{II}$ & : & & $\seq{\eta_0, q_0}$ & & $\seq{\eta_1, q_1}$ & $\cdots$,
\end{tabular}
\end{table}

\noindent
where $N_i\in\mathcal{C}$, $\eta_i\in[\omega_1, \omega_2)$ and $p_i$, $q_i\in\mathbb{P}$ for each $i<\omega$. Players must follow the following rules:
\begin{enumerate}[(a)]
  \item $p_i$ is $(N_i, \mathbb{P})$-strongly generic for each $i<\omega$.\label{rulea}
  \item $p_0\restrict N_0\leq_{\mathbb{P}}p$, and $p_{i+1}\restrict N_{i+1}\leq_{\mathbb{P}}q_i$ for each $i<\omega$.\label{ruleb}
  \item $N_{i+1}\succ^s_{\eta_i}N_i$ for each $i<\omega$.\label{rulec}
  \item $q_i\leq_\mathbb{P}p_i$ for each $i<\omega$.\label{ruled}
\end{enumerate}
Note that (\ref{rulea})--(\ref{rulec}) are required to Player $\mathrm{I}$, whereas (\ref{ruled}) is to Player $\mathrm{II}$. Player $\mathrm{I}$ wins if and only if he successfully finished his $\omega$ turns without becoming unable to make a legal move on the way.
\end{dfn}

The following lemma is an analogue of Velickovic \cite[Lemma 3.7]{velickovic92:_forcin}\footnote{The referee pointed out that an argument to some extent similar can also be found in Gitik \cite[Proof of Theorem 1.1]{Gitik1985:_GITNSO}.}.

\begin{lma}\normalfont\label{lma:I_wins}
Player $\mathrm{I}$ has a winning strategy for $\mathfrak{G}(\mathfrak{A}, p)$.
\end{lma}
\proof
Note that $\mathfrak{G}(\mathfrak{A}, p)$ is open to Player $\mathrm{II}$, and thus is determined. So it is enough to show that Player $\mathrm{II}$ does not have a winning strategy. Let $\upsilon$ be any strategy for Player $\mathrm{II}$.

By a simultaneous induction on $\omega_1\cdot n+\gamma$ (for $n<\omega$ and $\gamma<\omega_1$), we will choose $N^n_\gamma\in\mathcal{C}$, $r^n_\gamma$, $p^n_\gamma$, $q^n_\gamma\in\mathbb{P}$, $\eta^n_\gamma\in\omega_2$ and a partial play $R^n_\gamma$ of $\mathfrak{G}(\mathfrak{A}, p)$ so that the following requirements are satisfied: For each $m<\omega$ and $\xi<\omega_1$,
\begin{enumerate}[(i)]
\item $N^m_\xi \ni\seq{N^m_\zeta, r^m_\zeta, p^m_\zeta, q^m_\zeta, \eta^m_\zeta, R^m_\zeta}$ if $\xi=\zeta+1$.\label{sec6:Ninc}
\item $N^m_\xi \ni\seq{N^l_\zeta, r^l_\zeta, p^l_\zeta, q^l_\zeta, \eta^l_\zeta, R^l_\zeta\mid\zeta<\omega_1}$ if $m=l+1$.\label{sec6:Nexpand}
\item $N^m_\xi=\bigcup\{N^m_\zeta\mid\zeta<\xi\}$ if $\xi$ is limit.\label{sec6:Nconti}
\item $r^m_\xi\leq_\mathbb{P}q^m_\zeta$ if $\xi=\zeta+1$.\label{sec6:rleqq}
\item $r^m_\xi=\bigwedge\{r^m_\zeta\mid\zeta<\xi\}$ if $\xi$ is limit.\label{sec6:rconti}
\item $p^m_\xi=\tau(\xi, r^m_\xi)$.\label{sec6:ptaur}
\item $p^m_\xi$ is $(N^m_\xi, \mathbb{P})$-strongly generic and $p^m_\xi\restrict N^m_\xi=r^m_\xi$.\label{sec6:pgen}
\item $q^m_\xi\leq_\mathbb{P}p^m_\xi$.\label{sec6:qleqp}
\item $R^m_\xi$ is a part of a play of $\mathfrak{G}(\mathfrak{A}, p)$ where Player $\mathrm{II}$ plays according to $\upsilon$, ending with Player $\mathrm{I}$'s move $\seq{N^m_\xi, p^m_\xi}$ followed by Player $\mathrm{II}$'s move $\seq{\eta^m_\xi, q^m_\xi}$.\label{sec6:Rplay}
\item $R^m_\xi$ extends $R^l_{\delta^m_\xi}$ if $m=l+1$, where $\delta^m_\xi$ denotes $N^m_\xi\cap\omega_1$.\label{sec6:Rext}
\end{enumerate}

Let $n<\omega$ and $\gamma<\omega_1$, and suppose that $$\seq{N^m_\xi, r^m_\xi, p^m_\xi, q^m_\xi, \eta^m_\xi, R^m_\xi\mid\omega_1\cdot m+\xi<\omega_1\cdot n+\gamma}$$ is already defined and satisfy (i)--(ix) for every pair of $m$ and $\xi$ such that $\omega_1\cdot m+\xi<\omega_1\cdot n+\gamma$. We will define $N^n_\gamma$, $r^n_\gamma$, $p^n_\gamma$, $q^n_\gamma$, $\eta^n_\gamma$ and $R^n_\gamma$ so that (i)--(ix) for $m=n$ and $\xi=\gamma$ hold. Throughout the cases below, for example, \lq\lq(i) for $m=n$ and $\xi=\gamma$\rq\rq\ is simply expressed as \lq\lq(i)\rq\rq. Other combinations of $m$ and $\xi$, which appear as induction hypotheses, are explicitly specified.

\smallskip\noindent
\underline{Case 1} $n=\gamma=0$.

Pick $N^0_0\in\mathcal{C}$ so that $p\in N^0_0$. Pick an $(N^0_0, \mathbb{P})$-generic sequence $s^0_0$ beginning with $p$. By Lemma \ref{lma:gseqstrcl}, $s^0_0$ has a common extension in $\mathbb{P}$. Now set $r^0_0=\bigwedge s^0_0$ and $p^0_0=\tau(0, r^0_0)$. Then we have (\ref{sec6:ptaur}) and by Lemma \ref{lma:strgen} we have (\ref{sec6:pgen}), and the latter assures that $\seq{N^0_0, p^0_0}$ forms a legal $0$-th move of Player $\mathrm{I}$ in $\mathfrak{G}(\mathfrak{A}, p)$. Set $\seq{\eta^0_0, q^0_0}=\upsilon(\seq{N^0_0, p^0_0})$ and $R^0_0=\seq{\seq{N^0_0, p^0_0}, \seq{\eta^0_0, q^0_0}}$. Then (\ref{sec6:qleqp})and (\ref{sec6:Rplay}) are satisfied.

\smallskip\noindent
\underline{Case 2} $n=0$ and $\gamma=\zeta+1$.

Pick $N^0_\gamma\in\mathcal{C}$ so that $\seq{N^0_\zeta, r^0_\zeta, p^0_\zeta, q^0_\zeta, \eta^0_\zeta, R^0_\zeta}\in N^0_\gamma$. Then (\ref{sec6:Ninc}) is satisfied. Pick an $(N^0_\gamma, \mathbb{P})$-generic sequence $s^0_\gamma$ beginning with $q^0_\zeta$. By Lemma \ref{lma:gseqstrcl}, $s^0_\gamma$ has a common extension in $\mathbb{P}$. Now set $r^0_\gamma=\bigwedge s^0_\gamma$, $p^0_\gamma=\tau(\gamma, r^0_\gamma)$, $\seq{\eta^0_\gamma, q^0_\gamma}=\upsilon(\seq{N^0_\gamma, p^0_\gamma})$ and $R^0_\gamma=\seq{\seq{N^0_\gamma, p^0_\gamma}, \seq{\eta^0_\gamma, q^0_\gamma}}$. Then (\ref{sec6:rleqq}) is clear, and (\ref{sec6:ptaur})--(\ref{sec6:Rplay}) are obtained in the same way as in Case 1.

\smallskip\noindent
\underline{Case 3} $n=0$ and $\gamma$ is limit.

Let $N^0_\gamma=\bigcup\{N^0_\zeta\mid\zeta<\gamma\}$ and $r^0_\gamma=\bigwedge\{r^0_\zeta\mid\zeta<\gamma\}$. Then define $p^0_\gamma$, $q^0_\gamma$, $\eta^0_\gamma$ and $R^0_\gamma$ in the same way as in Case 2. By (\ref{sec6:Ninc}) and (\ref{sec6:Nconti}) for $m=0$ and $\xi<\gamma$, $\seq{N^0_\zeta\mid\zeta<\gamma}$ forms a $\subseteq$-increasing continuous chain of elements of $\mathcal{C}$, and thus we have $N^0_\gamma\in\mathcal{C}$ and (\ref{sec6:Nconti}). By (\ref{sec6:rleqq}), (\ref{sec6:rconti}), (\ref{sec6:ptaur}) and (\ref{sec6:qleqp}) for $m=0$ and $\xi<\gamma$  we have that
$$
\begin{matrix}
\mathrm{I}: & r^0_0 & r^0_1 & \cdots &&r^0_{\omega+1}& \cdots\\
\mathrm{II}: & \hspace{36pt} p^0_0 & \hspace{36pt} p^0_1 & \cdots & p^0_\omega &\hspace{36pt} p^0_{\omega+1} & \cdots
\end{matrix}
\quad\text{($\gamma$ turns)}
$$
forms a part of a play of $G_{\omega_1+1}(\mathbb{P})$ where Player $\mathrm{II}$ plays according to $\tau$, and thus we have $r^0_\gamma\in\mathbb{P}$ and (\ref{sec6:rconti}). (\ref{sec6:ptaur}) is clear by the construction. Pick a strictly increasing sequence $\seq{\gamma_j\mid j<\omega}$ of ordinals converging to $\gamma$. Then by the above observation we have $r^0_\gamma=\bigwedge\{p^0_{\gamma_j}\mid j<\omega\}$. Now by (\ref{sec6:Ninc}) and (\ref{sec6:pgen}) for $m=0$ and $\xi<\gamma$ and (\ref{sec6:Nconti}) for $m=0$ and $\xi=\gamma$, we have that $\seq{p^0_{\gamma_j}\mid j<\omega}$ is an $(N^m_\gamma, \mathbb{P})$-generic sequence, and thus by Lemma \ref{lma:strgen} we have (\ref{sec6:pgen}). (\ref{sec6:qleqp}) and (\ref{sec6:Rplay}) are obtained in the same way as in the former cases.

\smallskip\noindent
\underline{Case 4} $n=l+1$ and $\gamma=0$.

First note that by (\ref{sec6:rleqq}), (\ref{sec6:rconti}), (\ref{sec6:ptaur}) and (\ref{sec6:qleqp}) for $m=l$ we have that
$$
\begin{matrix}
\mathrm{I}: & r^l_0 & r^l_1 & \cdots &&r^l_{\omega+1}& \cdots\\
\mathrm{II}: & \hspace{36pt} p^l_0 & \hspace{36pt} p^l_1 & \cdots & p^l_\omega &\hspace{36pt} p^l_{\omega+1} & \cdots
\end{matrix}
\quad\text{($\omega_1$ turns)}
$$
forms a part of a play of $G_{\omega_1+1}(\mathbb{P})$ where Player $\mathrm{II}$ plays according to $\tau$. Therefore $\seq{r^l_\zeta\mid\zeta<\omega_1}$ has a common extension in $\mathbb{P}$. Now pick $N^n_0\in\mathcal{C}$ so that $\seq{N^l_\zeta, r^l_\zeta, p^l_\zeta, q^l_\zeta, \eta^l_\zeta, R^l_\zeta\mid \zeta<\omega_1}\in N^n_0$. This gives (\ref{sec6:Nexpand}). By elementarity there is a common extension $q_l\in\mathbb{P}$ of $\seq{r^l_\zeta\mid\zeta<\omega_1}$ in $N^n_0$. Pick an $(N^n_0, \mathbb{P})$-generic sequence $s^n_0$ beginning with $q_l$. By Lemma \ref{lma:gseqstrcl}, $s^n_0$ has a common extension in $\mathbb{P}$. Now set $r^n_0=\bigwedge s^n_0$ and  $p^n_0=\tau(0, r^n_0)$.  Then we have (\ref{sec6:ptaur}) and by Lemma \ref{lma:strgen} we have (\ref{sec6:pgen}). By (\ref{sec6:Rplay}) for $m=l$ and $\xi=\delta ^n_0$, $R^l_{\delta^n_0}$ is a part of a play of $\mathfrak{G}(\mathfrak{A}, p)$ where Player $\mathrm{II}$ plays according to $\upsilon$, ending with Player $\mathrm{I}$'s move $\seq{N^l_{\delta^n_0}, p^l_{\delta^n_0}}$ followed by Player $\mathrm{II}$'s move $\seq{\eta^l_{\delta^n_0}, q^l_{\delta^n_0}}$.
We show that $\seq{N^n_0, p^n_0}$ is a legal move of Player $\mathrm{I}$ following $R^l_{\delta^n_0}$.
(\ref{sec6:pgen}) assures that the rule (\ref{rulea}) is satisfied. By the above construction and (\ref{sec6:rleqq}) for $m=l$ and $\xi=\delta^n_0+1$ we have that $q^l_{\delta^n_0}\geq_{\mathbb{P}}r^l_{\delta^n_0+1}\geq_{\mathbb{P}}q_l\geq_{\mathbb{P}}r^n_0=p^n_0\restrict N^n_0$. This assures that the rule (\ref{ruleb}) is satisfied.
For the rule (\ref{rulec}), first note that by (\ref{sec6:Ninc}) and (\ref{sec6:Nconti}) for $m=l$ and $\xi<\omega_1$, $\seq{N^l_\zeta\mid\zeta<\omega_1}$ is a $\subseteq$-continuous strictly increasing chain, and thus $\eta^l:=\omega_2\cap\bigcup\{N^l_\zeta\mid\zeta<\omega_1\}$ is an ordinal below $\omega_2$. Since $\seq{N^l_\zeta\mid\zeta<\omega_1}\in N^n_0$, for each $\zeta<\delta^n_0$ it holds that $N^l_\zeta\in N^n_0$, and thus we have that $N^l_{\delta^n_0}=\bigcup\{N^l_\zeta\mid\zeta<\delta^n_0\}\subseteq N^n_0$. On the other hand, by elementarity, for each $\alpha\in\eta^l\cap N^n_0$ there exists $\zeta\in\omega_1\cap N^n_0=\delta^n_0$ such that $\alpha\in N^l_\zeta\subseteq N^l_{\delta^n_0}$. Moreover, since $\eta^l\in N^n_0\setminus N^l_{\delta^n_0}$ we have $N^l_{\delta^n_0}\cap\omega_2\not=N^n_0\cap\omega_2$. Therefore we have $N^l_{\delta^n_0}\prec^s_{\eta^l} N^n_0$, in particular $N^l_{\delta^n_0}\prec^s_{\eta^l_{\delta^n_0}} N^n_0$, since $\eta^l_{\delta^n_0}\in N^l_{\delta^n_0+1}$ holds by (\ref{sec6:Ninc}) for $m=l$ and $\xi=\delta^n_0+1$ and thus $\eta^l_{\delta^n_0}<\eta^l$ holds. Now let $\seq{\eta^n_0, q^n_0}=\upsilon(\concat{R^m_{\delta^n_0}}{\seq{\seq{N^n_0, p^n_0}}})$ and $R^n_0=\concat{R^m_{\delta^n_0}}{\seq{\seq{N^n_0, p^n_0},\seq{\eta^n_0, q^n_0}}}$. Now (\ref{sec6:qleqp}), (\ref{sec6:Rplay}) and (\ref{sec6:Rext}) are clear.

\smallskip\noindent
\underline{Case 5} $n=l+1$ and $\gamma=\zeta+1$.

Pick $N^n_\gamma\in\mathcal{C}$ so that $\seq{N^n_\zeta, r^n_\zeta, p^n_\zeta, q^n_\zeta, \eta^n_\zeta, R^n_\zeta}\in N^n_\gamma$. Pick an $(N^n_\gamma, \mathbb{P})$-generic sequence $s^n_\gamma$ beginning with $q^n_\zeta$. Then set $r^n_\gamma=\bigwedge s^n_\gamma$, $p^n_\gamma=\tau(\gamma, r^n_\gamma)$, $\seq{\eta^n_\gamma, q^n_\gamma}=\upsilon(\concat{R^l_{\delta^n_\gamma}}{\seq{\seq{N^n_\gamma, p^n_\gamma}}})$ and $R^n_\gamma=\concat{R^l_{\delta^n_\gamma}}{\seq{\seq{N^n_\gamma, p^n_\gamma},\seq{\eta^n_\gamma, q^n_\gamma}}}$. (\ref{sec6:Ninc}) is clear, and this and (\ref{sec6:Nexpand}) for $m=n$ and $\xi=\zeta$ imply (\ref{sec6:Nexpand}). (\ref{sec6:rleqq}) and (\ref{sec6:ptaur})--(\ref{sec6:Rext}) are obtained in the same way as in the former cases.

\smallskip\noindent
\underline{Case 6} $n=l+1$ and $\gamma$ is limit.

Let $N^n_\gamma=\bigcup\{N^n_\zeta\mid\zeta<\gamma\}$, $r^n_\gamma=\bigwedge\{r^n_\zeta\mid\zeta<\gamma\}$ and define $p^n_\gamma$, $q^n_\gamma$, $\eta^n_\gamma$ and $R^n_\gamma$ in the same way as in Case 5. Then (\ref{sec6:Nexpand}), (\ref{sec6:Nconti}) and (\ref{sec6:rconti})--(\ref{sec6:Rext}) are obtained in the same way as in the former cases.

Now for each $n<\omega$ let $C_n:=\{\zeta<\omega_1\mid\delta^n_\zeta=\zeta\}$. For each $n<\omega$, by (\ref{sec6:Ninc}) and (\ref{sec6:Nconti}) we have that $\seq{\delta^n_\zeta\mid\zeta<\omega_1}$ is a strictly increasing continuous sequence, and thus $C_n$ is a club subset of $\omega_1$. Therefore $C=\bigcap\{C_n\mid n<\omega\}$ is nonempty. Pick $\gamma\in C$. Then $\delta^n_\gamma=\gamma$ holds for every $n<\omega$, and thus by (\ref{sec6:Rext}) $\bigcup\{R^n_\gamma\mid n<\omega\}$ forms a full play in $\mathfrak{G}(\mathfrak{A}, p)$ where Player $\mathrm{II}$ plays according to $\upsilon$. This shows that $\upsilon$ is not a winning strategy.
\qed (Lemma \ref{lma:I_wins})

\begin{proof}[Proof of Lemma \ref{lma:stationary}]
Suppose $\dot{Z}$ is as in the assumption of the lemma.
Let $r$ be any condition in $\mathbb{S}$, and $\dot{f}$ any $\mathbb{S}$-name satisfying 
$$\force_\mathbb{S}\text{\lq\lq$\dot{f}:{}^{<\omega}\check{(\omega^V_2)}\to\check{(\omega^V_2)}$.\rq\rq}$$
Let $\theta$ be an uncountable regular cardinal satisfying $\{\mathbb{S}, \tau, \dot{f}\}\in H_\theta$, and set $\mathfrak{A}=\seq{H_\theta, \in, \mathbb{P}, \mathbb{S}, \tau, \dot{f}, \{r\}}$. To have the conclusion of the lemma, it is enough to show that there exists a countable $N\prec\mathfrak{A}$ and an $(N, \mathbb{S})$-generic $s$ extending $r$ such that $$s\force_{\mathbb{S}}\text{\lq\lq$\dot{Z}(\check{\eta})\nsubseteq\check{N}$,\rq\rq}$$ where $\eta$ denotes $\sup(N\cap\omega^V_2)$.

Let $p=\mathrm{lm}(r)\in\mathbb{P}$. By Lemma \ref{lma:I_wins} Player $\mathrm{I}$ has a winning strategy $\rho$ for the game $\mathfrak{G}(\mathfrak{A}, p)$. We will consider two plays of $\mathfrak{G}(\mathfrak{A}, p)$, named Play $A$ and $B$, played simultaneously in the way described below.

Moves in these plays are denoted as follows:

\begin{table}[h]
\begin{tabular}{cccccccc}
\underline{Play $A$}&&&&&&&\\
& $\mathrm{I}$ & : & $\seq{N^A_0, p^A_0}$ & & $\seq{N^A_1, p^A_1}$ & & $\cdots$\\
& $\mathrm{II}$ & : & & $\seq{\eta^A_0, q^A_0}$ & & $\seq{\eta^A_1, q^A_1}$ & $\cdots$\\
\underline{Play $B$}&&&&&&&\\
& $\mathrm{I}$ & : & $\seq{N^B_0, p^B_0}$ & & $\seq{N^B_1, p^B_1}$ & & $\cdots$\\
& $\mathrm{II}$ & : & & $\seq{\eta^B_0, q^B_0}$ & & $\seq{\eta^B_1, q^B_1}$ & $\cdots$
\end{tabular}
\end{table}

In both plays Player $\mathrm{I}$ plays according to $\rho$. Player $\mathrm{II}$ chooses her moves as follows:
$$
\begin{cases}
\seq{\eta^A_n, q^A_n}=\seq{\sup (N^B_n\cap\omega_2), \tau(n, p^B_n)},\\
\seq{\eta^B_n, q^B_n}=\seq{\sup (N^A_{n+1}\cap\omega_2), p^A_{n+1}}.
\end{cases}
$$

By the construction we have:
\begin{enumerate}[(i)]
  \item $N^A_0=N^B_0$.\label{item:n0}
  \item $N^A_0\prec^s_{\sup(N^B_0\cap\omega_2)}N^A_1\prec^s_{\sup(N^B_1\cap\omega_2)}N^A_2\prec^s_{\sup(N^B_2\cap\omega_2)}\cdots$.\label{item:na}
  \item $N^B_0\prec^s_{\sup(N^A_1\cap\omega_2)}N^B_1\prec^s_{\sup(N^A_2\cap\omega_2)}N^B_2\prec^s_{\sup(N^A_3\cap\omega_2)}\cdots$.\label{item:nb}
  \item $p\geq_\mathbb{P}p^A_0\restrict N^A_0=p^B_0\restrict N^B_0\geq_\mathbb{P}p^A_0=p^B_0$.\label{item:p0}
  \item For every $n<\omega$, $p^A_n$ is $(N^A_n, \mathbb{P})$-strongly generic and $p^B_n$ is $(N^B_n, \mathbb{P})$-strongly generic.\label{item:stgen}
  \item For every $n<\omega$, it holds that
 $$
    p^B_n\geq_\mathbb{P}\tau(n, p^B_n)=q^A_n\geq_\mathbb{P}p^A_{n+1}\restrict N^A_{n+1}\geq_\mathbb{P}p^A_{n+1}=q^B_n\geq_\mathbb{P}p^B_{n+1}\restrict N^B_{n+1}\geq_\mathbb{P}p^B_{n+1}.
 $$\label{item:ineq}
\end{enumerate}

By (\ref{item:ineq}), $\seq{p^B_n, q^A_n\mid n<\omega}$ forms a part of a play in $G_{\omega_1+1}(\mathbb{P})$ where Player $\mathrm{II}$ plays along $\tau$, and thus we have
\begin{equation}\label{eq:q}
q:=\bigwedge\{p^B_n\mid n<\omega\}=\bigwedge\{p^A_n\restrict N^A_n\mid n<\omega\}=\bigwedge\{p^B_n\restrict N^B_n\mid n<\omega\}\in\mathbb{P}.
\end{equation}
Let $N^A=\bigcup\{N^A_n\mid n<\omega\}$, $N^B=\bigcup\{N^B_n\mid n<\omega\}$. By (\ref{item:na}) and  (\ref{item:nb}), $N^A$, $N^B$ are countable elementary submodels of $\mathfrak{A}$. Moreover, by (\ref{item:n0})--(\ref{item:nb}) we have
\begin{itemize}
  \item $N^A\cap\omega_1=N^A_0\cap\omega_1=N^B_0\cap\omega_1=N^B\cap\omega_1$ and
  \item $\sup (N^A\cap\omega_2)=\sup (N^B\cap\omega_2)$.
\end{itemize}
We will denote these two ordinals as $\delta$ and $\eta$.

Note also that, by (\ref{item:na}), (\ref{item:nb}), (\ref{item:stgen}) and (\ref{item:ineq}), $\seq{p^A_n\mid n<\omega}$ is an $(N^A, \mathbb{P})$-generic sequence and $\seq{p^B_n\mid n<\omega}$ is an $(N^B, \mathbb{P})$-generic sequence. Therefore by Lemma \ref{lma:strgen} we have
\begin{equation}
q=q\restrict N^A=q\restrict N^B.\label{eqn:qqnaqnb}
\end{equation}

Now let $q'=\tau(\delta, q)$. Since $N^A\cap N^B\cap\omega_2=N^A_0\cap\omega_2$ is bounded in $\eta$ by (\ref{item:na}) and (\ref{item:nb}), it holds that
$$
q'\force_{\mathbb{P}}\text{\lq\lq$\dot{Z}(\check{\eta})\nsubseteq\check{N^A}\lor\dot{Z}(\check{\eta})\nsubseteq\check{N^B}$.\rq\rq}
$$
Therefore, without loss of generality we may assume that there exists $q''\leq_{\mathbb{P}}q'$ such that
\begin{equation}
q''\force_{\mathbb{P}}\text{\lq\lq$\dot{Z}(\check{\eta})\nsubseteq\check{N^A}$.\rq\rq}\label{eqn:q''forces}
\end{equation}

Now since $\mathrm{lm}(r)=p\geq_\mathbb{P}q$ and $r\in N^A$ holds, repeatedly using Lemma \ref{lma:onestep}, we obtain an $(N^A, \mathbb{S})$-generic sequence $\seq{r_n\mid n<\omega}$ such that $r_0=r$ and that $\mathrm{lm}(r_n)\geq_\mathbb{P}q$ holds for every $n<\omega$. Let us denote $\bigcup\{r_n\mid n<\omega\}=\seq{a_\gamma, b_\gamma\mid\gamma<\overline{\delta}}$. By an easy density argument we have $\overline{\delta}=N^A\cap\omega_1=\delta$. Since $\mathrm{lm}$ is a projection, $\seq{\mathrm{lm}(r_n)\mid n<\omega}$ is an $(N^A, \mathbb{P})$-generic sequence, and thus by Lemma \ref{lma:strgen} and (\ref{eqn:qqnaqnb}) we have
\begin{equation}
\bigwedge\{b_\gamma\mid\gamma<\delta\}=\bigwedge\{\mathrm{lm}(r_n)\mid n<\omega\}=q\restrict N^A=q.\label{eqn:bgqN}
\end{equation}
Therefore
$$
s=\concat{(\bigcup\{r_n\mid n<\omega\})}{\seq{q, q''}}
$$
forms an $(N^A, \mathbb{S})$-strongly generic condition satisfying $\mathrm{lm}(s)\leq_\mathbb{P}q''$. By (\ref{eqn:q''forces}) and absoluteness we have
$$
s\force_\mathbb{S}\text{\lq\lq$\dot{Z}(\check{\eta})\nsubseteq\check{N^A}$.\rq\rq}
$$
This shows that $s$ is what we need, and completes the proof of the lemma.
\end{proof}

\section*{Acknowledgments}
The author would like to express his gratitude to David Asper\'o, Yo Matsubara, Tadatoshi Miyamoto, Hiroshi Sakai and Toshimichi Usuba for their helpful comments.

\bibliographystyle{plain}
\bibliography{locco}

\begin{thebibliography}{10}

\bibitem{baumgartner:_iterated}
James Baumgartner.
\newblock Iterated forcing.
\newblock In A.R.D. Mathias, editor, {\em Surveys in Set Theory}, pages 1--59.
  Cambridge Univ. Press, 1979.

\bibitem{caicedo_velickovic06:}
Andr\'es~Eduardo Caicedo and Boban Velickovic.
\newblock The bounded proper forcing axiom and well orderings of the reals.
\newblock {\em Mathematical Research Letters}, 13:393--408, 2006.

\bibitem{cummings:_iterated}
James Cummings.
\newblock Iterated {F}orcing and {E}lementary {E}mbeddings.
\newblock In M.~Foreman and A.~Kanamori, editors, {\em Handbook of Set Theory,
  Volume II}, pages 775--883. Springer, 2010.

\bibitem{cummings:_martinsmaximum}
James Cummings and Menachem Magidor.
\newblock Martin's {M}aximum and {W}eak {S}quare.
\newblock {\em Proceedings of the American Mathematical Society},
  139:3339--3348, 2011.

\bibitem{foreman83:_games_boolean}
Matthew Foreman.
\newblock Games played on {B}oolean algebras.
\newblock {\em Journal of Symbolic Logic}, 48:714--723, 1983.

\bibitem{foreman88}
Matthew Foreman, Menachem Magidor, and Saharon Shelah.
\newblock Martin's {M}aximum, saturated ideals, and nonregular ultrafilters
  {I}.
\newblock {\em Annals of Mathematics}, 127:1--47, 1988.

\bibitem{sdfriedman:_class}
Sy~D. Friedman.
\newblock {\em Fine structure and class forcing}.
\newblock Walter de Gruyter \& Co., Berlin, 2000.

\bibitem{Gitik1985:_GITNSO}
Moti Gitik.
\newblock Nonsplitting subset of ${P}_\kappa(\kappa^+)$.
\newblock {\em Journal of Symbolic Logic}, 50(4):881--894, 1985.

\bibitem{ishiuyoshinobu02}
Tetsuya Ishiu and Yasuo Yoshinobu.
\newblock Directive trees and games on posets.
\newblock {\em Proceedings of the American Mathematical Society},
  130(5):1477--1485, 2002.

\bibitem{jech:_game}
Thomas Jech.
\newblock A game theoretic property of {B}oolean algebras.
\newblock In A.~Macintyre, L.~Pacholski, and J.~Paris, editors, {\em Logic
  Colloquium '77}, pages 135--144. North-Holland, 1978.

\bibitem{koenig04:_fragm_maxim}
Bernhard K\"onig and Yasuo Yoshinobu.
\newblock Fragments of {M}artin's {M}aximum in generic extensions.
\newblock {\em Mathematical Logic Quarterly}, 50:297--302, 2004.

\bibitem{ky12:_kurep_namba}
Bernhard K\"onig and Yasuo Yoshinobu.
\newblock Kurepa-trees and {N}amba forcing.
\newblock {\em Journal of Symbolic Logic}, 77(4):1281--1290, 2012.

\bibitem{larson00:_separ}
Paul Larson.
\newblock Separating stationary reflection principles.
\newblock {\em Journal of Symbolic Logic}, 65:247--258, 2000.

\bibitem{magidor:_jerusalem}
Menachem Magidor.
\newblock Lectures on weak square principles and forcing axioms.
\newblock Jerusalem Logic Seminar, Summer 1995.

\bibitem{miyamoto:_faandcc}
Tadatoshi Miyamoto.
\newblock A {F}orcing {A}xiom and {C}hang's {C}onjecture.
\newblock {\em RIMS Kokyuroku}, 1790:51--59, 2012.

\bibitem{sakai:_ccandws}
Hiroshi Sakai.
\newblock Chang's {C}onjecture and weak square.
\newblock {\em Archive for Mathematical Logic}, 52:29--45, 2013.

\bibitem{schimmerling:_onewoodin}
Ernest Schimmerling.
\newblock {C}ombinatorial principles in the core model for one {W}oodin
  cardinal.
\newblock {\em Annals of Pure and Applied Logic}, 74:153--201, 1995.

\bibitem{PFAnote}
Stevo Todorcevic.
\newblock A note on the {P}roper {F}orcing {A}xiom.
\newblock {\em Contemporary Mathematics}, 31:209--218, 1984.

\bibitem{usuba:_notesonmiyamotofa}
Toshimichi Usuba.
\newblock Notes on {M}iyamoto's {F}orcing {A}xioms.
\newblock {\em RIMS Kokyuroku}, 1790:60--64, 2012.

\bibitem{velickovic92:_forcin}
Boban Velickovic.
\newblock Forcing axioms and stationary sets.
\newblock {\em Advances in Mathematics}, 94:256--284, 1992.

\bibitem{velleman_strategy}
Dan Velleman.
\newblock On a generalization of {J}ensen's $\square_\kappa$, and strategic
  closure of partial orders.
\newblock {\em Journal of Symbolic Logic}, 48:1046--1052, 1983.

\bibitem{yoshinobu13:_oper}
Yasuo Yoshinobu.
\newblock Operations, climbability and the proper forcing axiom.
\newblock {\em Annals of Pure and Applied Logic}, 164(7-8):749--762, 2013.

\end{thebibliography}
\end{document}